\documentclass[11pt,uplatex]{amsart}

\usepackage{amsmath,amssymb,amsthm}
\usepackage{enumerate}
\usepackage[all]{xy}
\usepackage[dvips]{color}
\usepackage{version}
\usepackage{graphicx}

\newtheorem{Thm}{Theorem}[section]
\newtheorem{Lem}[Thm]{Lemma}
\newtheorem{Fact}[Thm]{Fact}
\newtheorem{Cor}[Thm]{Corollary}

\newtheorem{Prop}[Thm]{Proposition}

\newtheorem{Conj}[Thm]{Conjecture}
\newtheorem{``Conj"}[Thm]{``Conjecture"}
\newtheorem{Predict}[Thm]{Prediction}

\newtheorem{Ques}{Question}
\newtheorem{Fib}{Fibrations}

\theoremstyle{remark}
\newtheorem{Rem}[Thm]{Remark}

\theoremstyle{definition}
\newtheorem{Def}[Thm]{Definition}
\newtheorem{Setup}[Thm]{Setup}

\newtheorem*{ack}{Acknowledgements}

\newcommand{\R}{\mathbb{R}}
\newcommand{\C}{\mathbb{C}}

\newcommand{\ZZ}{\mathbb{Z}}

\newcommand{\X}{\mathcal{X}}
\newcommand{\cL}{\mathcal{L}}

\newcommand{\Aut}{\mathop{\mathrm{Aut}}\nolimits}

\usepackage{marginnote}

\begin{document}

\title[Hybrid SYZ fibration and finite symmetries]
{Special Lagrangian fibrations,
Berkovich retraction, and
crystallographic groups}

\author{Keita Goto, Yuji Odaka}
\date{\today}

\maketitle


\begin{abstract}
    We explicitly construct
    special Lagrangian fibrations on
    finite quotients of maximally
    degenerating abelian varieties, and
    glue these to the Berkovich retraction to form a  ``hybrid" fibration.
    We also study their symmetries  explicitly  which can be regarded as
     crystallographic groups.
    In particular, a conjecture of Kontsevich-Soibelman \cite[Conjecture 3]{KS}
    is solved at an enhanced level for
    finite quotients of abelian varieties in
    any dimension.
\end{abstract}


\section{Introduction}

We discuss the following two kinds of fibrations which appear around the so-called ``large complex structure limits" or ``maximal degenerations"
of K-trivial varieties
(a.k.a. Calabi-Yau varieties, in their most general sense).

\begin{Fib}\label{CSYZ}
Special Lagrangian fibrations from {\it complex} Calabi-Yau manifolds (often called  ``SYZ fibrations")
\end{Fib}
\begin{Fib} \label{NASYZ}
Berkovich type retractions of {\it Berkovich (non-archimedean)} space to the essential
skeleton (often called ``non-archimedean SYZ fibration").
\end{Fib}

Their historic origins are quite different.  The former dates back at least to Harvey-Lawson \cite[III]{HL}  in the context of calibrated  geometry. Later,
the paper by
Strominger-Yau-Zaslow \cite{SYZ},
shed light from the context of the mirror
symmetry and
expected that the fibration structure \ref{CSYZ}
exists for any ``maximal" degenerations. This fairly
nontrivial
prediction is occasionally referred to as (a part of)
``SYZ conjecture".

\begin{Conj}[the SYZ conjecture on special Lagrangian fibrations]\label{metricSYZ}
Take any flat proper family $f \colon \mathcal{X}^*\to \Delta^*=\{t\in \C\mid 0<|t|<1\}$
which is a maximal degeneration i.e.,
with the maximal unipotency index of the monodromy.
Also take any relatively ample line bundle $\mathcal{L}^*$ and
consider the famliy of Ricci-flat K\"ahler metrics $g_{\rm KE}(\mathcal{X}_t)$ on $\pi
^{-1}(t)=\mathcal{X}_t$ for $t\neq 0$,
whose K\"ahler classes are $c_1(\mathcal{L}_t:=\mathcal{L}|_{\mathcal{X}_t})$.

Then for any $|t|\ll 1$, there is a special Lagrangian fibration on $\mathcal{X}_t$
(with respect to a certain meromorphic relative section of $K_{\mathcal{X}^*/\Delta^*}$).
\end{Conj}

On the other hand,
the latter fibration (retraction)  first appeared in the  basic general theory of  non-archimedean
geometry \cite{Ber99}.
Note that in both situations,
affine manifolds with singularities
underly the fibrations. Such
a heuristic similarity of these two fibrations is first pointed out in
Kontsevich-Soibelman \cite{KS}, in which
both fibrations are regarded roughly as
``tropicalization".
The insight in particular lead to the following prediction (\cite[Conjecture 3]{KS}),
which was not originally quite precise.
\begin{Predict}[{\cite[Conjecture 3]{KS}}]\label{KS3}
For each maximal degeneration of polarized K-trivial varieties, as Conjecture \ref{metricSYZ},
the following two affine structures coincide:
\begin{itemize}
\item one affine structure obtained by the ``complex picture" through
the special Lagrangian fibrations or the collapse of
the Calabi-Yau metrics and
\item another affine structure underlying the non-archimedean SYZ fibration (the ``non-archimedean picture").
\end{itemize}
\end{Predict}

\cite[\S 5]{Goto} by the first author make the above prediction \ref{KS3} ($=$\cite[Conjecture 3]{KS}) rigorous and
proves it
for abelian varieties and  Kummer surfaces in an explicit manner, with a particular emphasis on the
non-archimedean SYZ fibration side.

We quickly summarize our main statements as follows,
but we leave
the details of the statements to later sections.
\begin{Thm}\label{summary}
For any maximally degenerating abelian varieties,
\begin{enumerate}
\item (Theorems \ref{SLAG.OO}, \ref{SLAG.ppav}) the SYZ conjecture \ref{metricSYZ} holds,
\item (Theorem \ref{Hybrid SYZ}) we can glue the obtained SYZ fibrations with the non-archimedean SYZ fibration
(to what we call the ``hybrid SYZ fibration").
\end{enumerate}
\end{Thm}
\begin{Cor}[Corollary 3.3]
For the above setup as Theorem \ref{summary},
Prediction \ref{KS3} holds.
\end{Cor}
\noindent
We note that
all the fibrations are made explicit in our discussions.
Moreover,
we generalize the above statements to
finite K-trivial quotients (of maximally degenerating
polarized abelian varieties)
under mild assumptions.

\vspace{3mm}
Here are more details of our organization of this paper.
In the next section, we start with rigorous
proofs of the SYZ conjecture \ref{metricSYZ}
for K-trivial surfaces, abelian varieties
of any dimension and their finite quotients (section \ref{SLAG.sec}).
Recall that already in \cite{OO}, Conjecture \ref{metricSYZ} was proven for the case of K3 surfaces ({\it op.cit} Chapter 4)
 and higher dimensional irreducible holomorphic symplectic varieties of
K3${}^{[n]}$-type and generalized Kummer varieties  ({\it op.cit} Chapter 8).

In section \ref{hyb.sec}, we show how to ``merge"  these two fibrations \eqref{CSYZ} and \eqref{NASYZ} above with different origins, to construct what we call a
{\it hybrid SYZ fibration}.
This provides some enhanced
answer to
\cite[Conjecture 3]{KS},
not only at the level of
integral affine structures.
Our construction is inspired
by
the recent technology of hybrid norms originated in \cite{Ber.Manin} and
re-explored in \cite{BJ}.
We prove
\cite[Conjecture 3]{KS} for K-trivial
finite quotients
of abelian varieties in arbitrary dimension,
generalizing \cite[\S 5]{Goto}.

In section \S \ref{sec.sym}, to explicitly deal with finite quotients of abelian varieties,
we explore the possible symmetry of
maximally degenerating abelian varieties
and its inducing symmetry on the
Gromov-Hausdorff limit,
which fits well to the previous sections.

Lastly, we added an appendix which show that
of the (complex)
Calabi-Yau metrics limit to the non-archimedean Calabi-Yau metric in the sense of
e.g., \cite{BFJ}, in a certain sense. This answers a question of Yang Li.

\vspace{2mm}

Let us collect some notations we use
throughout.

\subsection*{Notation}

\begin{itemize}
    \item $\Delta$ denotes the disk $\{t\in \C\mid |t|<1\}$,
    \item $\Delta^*$ denotes the punctured disk $\{t\in \C\mid 0<|t|<1 \}$,
    \item $\C((t))^{\rm mero}$
denotes the field of the germs of
meromorphic functions with possible pole only at the origin,
    \item The $t$-adic discrete valuation of
    $\C((t))^{\rm mero}$ is denoted as
    ${\rm val}_t \colon (\C((t))^{\rm mero}\setminus \{0\})\to \ZZ$
    \item $(\mathcal{X},\mathcal{L})\to
    \Delta$ (resp.,
    $(\mathcal{X}^*,\mathcal{L}^*)\to
    \Delta^*$)
    denotes a proper holomorphic map from a complex analytic space $\mathcal{X}$
    (resp., $\mathcal{X}^*$)
    together with an ample line bundle $\mathcal{L}$
    (resp., $\mathcal{L}^*$). We also denote as
    $(\mathcal{X},\mathcal{L})/
    \Delta$ (resp.,
    $(\mathcal{X}^*,\mathcal{L}^*)/
    \Delta^*$)
    \item Each fiber of $\mathcal{X}$ over $\Delta\ni t$
    is denoted as $\mathcal{X}_t$. The restriction of $\mathcal{L}$
    to $\mathcal{X}_t$ is denoted as $\mathcal{L}_t$.
    \item
    For the above pair $(\X^*,\cL^*)$, we  associate the smooth projective variety $X$ over $\C((t))^{\rm mero}$ with ample line bundle $L$.
    \item When $\mathcal{X}_t$s are abelian varieties,
    the type of polarization $\mathcal{L}_t$ is
    $(e_1,\cdots,e_g)\in \mathbb{Z}_{>0}^g$
    such that $e_i\mid e_{i+1}$ for any $(0<)i(<g)$.
    We associate a $g\times g$ diagonal matrix
    $E:={\rm diag}(e_1,\cdots,e_g)$.
    \item An {\it integral affine structure}
    (resp., {\it tropical affine structure})
    on a manifold of dimension $n$ means the maximal atlas of local coordinates whose transition functions are in ${\rm GL}(n,\mathbb{Z})\ltimes \mathbb{Z}^n$
    (resp., ${\rm GL}(n,\mathbb{Z})\ltimes \mathbb{R}^n$). This convention follows
    that of e.g. \cite{Gross}.
    \item After \cite{KS, Gross},
    {\it tropical (resp., integral) affine manifold $B$ with singularities} mean a manifold $B$
    and its
    closed subset $\Gamma$ which is locally finite union of submanifolds of codimension at least $2$, together with a
    tropical affine structure (resp., integral affine structure) on the open dense subset $B\setminus \Gamma$ of $B$.
    \end{itemize}

\begin{ack}

This work is partially supported by
supported by JSPS KAKENHI JP20J23401 for K.G. and
KAKENHI 16H06335 (Grant-in-Aid for Scientific Research (S)),
20H00112 (Grant-in-Aid for Scientific Research (A)),
21H00973 (Grant-in-Aid for Scientific Research (B))
and
KAKENHI 18K13389 (Grant-in-Aid for Early-Career Scientists) for Y.O.
\end{ack}


\section{Special Lagrangian fibration}\label{SLAG.sec}

\subsection{K-trivial surfaces case, after \cite{OO}}

This section explains one method
for partially proving Conjecture \ref{metricSYZ} by using
hyperK\"ahler rotation,
following \cite[\S 4]{OO}. The point is that,  although {\it loc.cit} focused on the case of K3 surfaces, the same method applies to the case of abelian surfaces.
In the next (sub)section,
we generalize to higher dimensional abelian varieties
by a more explicit different method.

\begin{Thm}\label{SLAG.OO}
Take an arbitrary maximally degenerating
family of polarized abelian surfaces (resp., polarized K3 surfaces possibly with ADE singularities)
$(\X|_{\Delta^*},\mathcal{L}|_{\Delta^*})$
over $\Delta^*$ with a fiber-preserving symplectic action of finite group
$H$ on $\X|_{\Delta^*}$ together with linearization on
$\mathcal{L}|_{\Delta^*}$ (e.g., $H$ can be even trivial or simple $\{\pm 1\}$-multiplication in
the case of abelian varieties).

We denote the quotient by $H$ as
$(\X'|_{\Delta^*},\mathcal{L}'|_{\Delta^*})\to \Delta^*$.
Then, the following hold:
\begin{enumerate}

\item \label{i} For any $t\in \Delta^*$ with $|t|\ll 1$,
there is a special Lagrangian fibration $f_t\colon \X_t \to \mathcal{B}_t$
with respect to the K\"ahler form $\omega_t$
of the flat K\"ahler metric $g_{\rm KE}(\X_t)$
with $[\omega_t]=c_1(\mathcal{L}_t)$
and the imaginary part ${\rm Im}(\Omega_t)$
of a certain holomorphic volume form
$(0\neq ) \Omega_t\in H^0(\X_t,\omega_{\X_t})$.
Here, $\mathcal{B}_t$ is a $2$-torus (resp., $S^2$)
and so are all fibers of $f_t$.
Note that $\omega_t$ and ${\rm Im}(\Omega_t)$ again induce tropical
affine structures on $\mathcal{B}_t$
(cf., e.g., \cite[Definition 1.1, \S 1]{Gross})
as $\nabla_{A}(t)$  and $\nabla_{B}(t)$ respectively,
as well as its McLean metric $g_t$. Below, we continue to assume $|t|\ll 1$.
In the next section \S \ref{hyb.sec},
we discuss how these $f_t$ glue to a family.

\item \label{ii}
Consider the obtained base associated with a
tropical affine structure
and a flat metric
$(\mathcal{B}_t,\nabla_A(t),\nabla_B(t),g_t)$ for $t\neq 0$. They converge to another
a $2$-torus (resp., $S^2$)
with the same additional structures
$(\mathcal{B}_0,\nabla_A(0),\nabla_B(0),g_0)$ in the natural sense, when $t\to 0$.

In this terminology, the Gromov-Hausdorff limit of $(\X_t,g_{\rm KE}(\X_t)/{\rm diam}(g_{\rm KE}(\X_t))^2)$ for $t\to 0$
coincides with $(\mathcal{B}_0,g_0)$. Here, ${\rm diam}(g_{\rm KE}(\X_t))$
refers to the diameter which is
used to rescale the metric to
that of diameter $1$ (as in
\cite{KS}).

\item \label{ab.iii}
The $H$-action on $\X_t$ preserves the fibers of $f_t$.
Thus, there is a natural induced action of $H$ on $\mathcal{B}_0$, which preserves the
three structures $\nabla_A(0),\nabla_B(0)$ and $g_0$.
The natural quotient of $f_t$ by $H$ denoted as
$f'_t\colon \X'_t\to \mathcal{B}'_t$ is again a special Lagrangian fibration
with respect to the descents of $\omega_t$ the holomorphic volume form
and $(0\neq ) \Omega_t\in H^0(\X_t,\omega_{\X_t})$ we chose above.

\item \label{ab.iv}
If $(\X|_{\Delta^*},\mathcal{L}|_{\Delta^*})$ is a family of
principally polarized abelian surfaces and $H$ is trivial, the tropical affine
structure $\nabla_A(0)$ on $\mathcal{B}_0$ is integral
(\cite[1.1]{Gross}) and its
integral points
consist of only $1$ point.
The corresponding Gram matrix of $g_0$ is the same $(cB(e_i,e_j))$ as appeared in \cite{Goto}.
Also the transition function of the integral basis of $\nabla_A(0)$ to that of $\nabla_B(0)$
is given by the same matrix $(cB(e_i,e_j))$.

\end{enumerate}
\end{Thm}

\begin{proof}
Our assertions above \eqref{i} and \eqref{ii} for K3 surfaces case are
proven in \cite[Chapter 4 (and 5,6 partially)]{OO}.

The proof of \eqref{i}, \eqref{ii} for the case for polarized abelian surfaces
and essentially easier and here we follow
the method of {\it loc.cit}. (In the next subsection, we give another proof.)
Hence, we below only sketch (review) the proof as a review and explain the differences with the original K3 surfaces case
in \cite[\S 4]{OO}.
More precisely speaking, we use the arguments of the proofs of
\cite[Theorems 4.11, 4.20]{OO} and other claims
on which they depend.

The (almost verbatim) change of basic setup is as follows.
For an abelian surface $X\simeq \mathbb{C}^2/\Lambda$ with a lattice $\Lambda(\simeq \mathbb{Z}^4)$,
set $\Lambda_{CT}:=\bigwedge^2_{\mathbb{Z}} \Lambda$.
(Here, the subscript CT stands for
complex torus.)
By the orientation on $\mathcal{X}_t$
induced by the complex structure,
we identify $\bigwedge^4_{\mathbb{Z}} \Lambda\simeq \mathbb{Z}$ which induces a lattice structure on
$\Lambda_{\rm CT}$ as isomorphic to $\textrm{I\hspace{-1pt}I}_{3,3}=U^{\oplus 3}$. Take a marking $\alpha$
of $H^2(\X_t,\mathbb{Z})$ for some $t$, and put $\lambda:=\alpha(c_1(\mathcal{L}_t))$,
and replace $\Lambda_{2d}$ of \cite[Chapter 4]{OO} by
$\Lambda_{pas}:=\lambda^{\perp}\subset
\Lambda_{CT}$ which has signature $(2,3)$.
The period domain for $(\X_t,\mathcal{L}_t)$ with weight $2$ is
$$\Omega({\Lambda_{pas}}):=\{[w]\in \mathbb{P}(\Lambda_{pas}\otimes \mathbb{C})
\mid (w,w)=0, (w,\bar{w})>0\},$$
which replaces $\Omega({\Lambda_{\rm K3}})$ of \cite[\S 4.1]{OO}.
By the accidental isomorphism ${\rm Sp}(4,\mathbb{R})/\{\pm 1\}\simeq
SO_0 (2,3)$, we can also identify the connected component of
$\Omega({\Lambda_{pas}})$ as the Siegel upper half space $\mathfrak{S}_2$
of degree $2$. Similarly, we also define and use the union of K\"ahler cones
$K\Omega$ (resp., $K\Omega^0$, $K\Omega^{e\ge 0}$)
as \cite[p.45-46]{OO}. Since the K\"ahler cones of
complex tori are just connected components of positive cones,
the definition of {\it loc.cit} works verbatim if we put $\Delta(\Lambda_{CT}):=\emptyset$, where $\Delta(\Lambda_{CT})$ is the set of simple roots.

Now we are ready to prove the above theorem.
The desired special Lagrangian fibration \eqref{i}
for polarized K3 surfaces family $(\mathcal{X}|_{\Delta^*}=\cup_t
\mathcal{X}_t,
\mathcal{L}|_{\mathcal{X}^*})$ case is obtained at
the bottom of \cite[p.47]{OO} (if we see $\pi_i$ as a map from $X_i$) during the proof of
\cite[4.20]{OO}.
Note that
the choice of a generator $\Omega_t\in H^0(\Omega_{X_t})$
for each $t$ is a priori up to constant multiplication but we
use the careful choice after
\cite[\S 4, especially (4.11)]{OO}.

More precisely, the proof (of \eqref{i})
uses ({\it loc.cit} Claim 6.12 which in turn follows from)
Fact 4.14 and Claim 4.18 of {\it loc.cit}. It applies a hyperK\"ahler rotation,
which also works for flat complex tori
because their holonomy groups are even trivial.

Proof of Claim 4.18 becomes much simpler for $2$-dimensional
complex torus case because we do not need to deal with the roots as ``$\Delta(X)^+$".

The only nontrivial difference exists for Fact 4.14 of \cite{OO}. Indeed, it does
{\it not} literally holds for
complex tori just because the morphism from complex $2$-tori to elliptic curves are not obtained as
pencils.
\begin{Fact}\label{ab.fib}
Let $X$ be a complex $2$-torus and let $e\in H^2(X,\mathbb{Z})$ a (primitive) isotropic.
If $e$ belongs to the closure of its K\"ahler cone, there exists a holomorphic fibration
to an elliptic curve $B$ as $X\twoheadrightarrow B$ whose
fiber class is $e$.
\end{Fact}
\begin{proof}[proof of Fact \ref{ab.fib}]
We take a line bundle $L$ with $c_1(L)=e$.
From \cite[\S 3.3]{BL}, $L$ is effective and represented by an elliptic curve $E$ by the
isotropicity and primitivity condition. If we replace $E$ by a translation of $E$ which
contains the origin of $X$, then $X\twoheadrightarrow B:=X/E$ is the desired morphism.
\end{proof}
Note that the fibers are only homologically identified as $e$, rather than
linear equivalence classes (as in \cite[4.14]{OO} for K3 surfaces).

The proof of \eqref{ii} is similar to the construction of
$X\to B$ of \cite[bottom of p.34-p.35]{OO}. The only difference again is that
we use above Fact \ref{ab.fib} instead of Fact 4.14 of {\it loc.cit}.
Then the proof is reduced to that of 4.22 of {\it loc.cit}, hence that of Chapter 5.
The arguments there work verbatim and is even greatly simplifiable. Indeed,
the main analytic difficulties in K3 surfaces case come from the presence of
(varying) ADE singularities, which destroys the simplest uniform $C^2$-estimates.
(In our case,
it is even possible to confirm 4.22 of
{\it loc.cit} explicitly. )

From here, we give the proofs of \eqref{ab.iii}, \eqref{ab.iv}.
To show \eqref{ab.iii}, we proceed to analyze the effect of the action of $H$.
Since the $H$-action on $\X_t$ is assumed to be symplectic i.e.,
the induced action on $H^0(\X_t, \Omega^2_{\X_t})$ is trivial, in particular
${\rm Im}(\Omega_{t})$
is preserved by the action of $H$. Therefore, the complex structure of the
hyperK\"ahler rotation of $\X_t$ as constructed in \cite[p.47]{OO}
is also preserved by the action. The action also preserves the
isotropic cohomology class $e$. Hence, from the construction of $f_t$,
the action preserves its fibers and the claims of \eqref{ab.iii} are proved.

For the proof of \eqref{ab.iv}, note that under appropriate symplectic basis
$v_1, w_1, v_2, w_2$ of $H^1(\X_t,\mathbb{Z})$ for the principal polarization $\langle -,-\rangle$ i.e.,
$\langle v_i, w_j\rangle=\delta_{i,j}$ and $\langle v_i, v_j\rangle =\langle w_i,w_j\rangle=0$,
$e$ is written as $v_1\wedge w_1$, $H^2(\mathcal{B}_t,\mathbb{Z})$ is written as $\mathbb{Z}(v_2\wedge
w_2)$,
$\lambda=v_1\wedge w_1+v_2\wedge w_2$. Then the direct computation shows
that the integral
affine structure $\nabla_A(0)$ on $\mathcal{B}_0\simeq \mathbb{R}^2/\mathbb{Z}^2$ is the
desired one. The last claim then follows from the definition of Legendre transform (cf., \cite[\S 3]{Hit}, \cite[\S 1]{Gross}). \end{proof}

\begin{Rem}
As we noted in the proof, see
\cite[\S 4, especially (4.11)]{OO} for the actual precise choice of
$\Omega_t\in H^0(\Omega_{\mathcal{X}_t})$ on which the proof depends. \end{Rem}

\begin{Rem}[Effect of monodromy on
the hyperKahler rotation]
Both the constructions of  $f_t\colon \X_t \to \mathcal{B}_t$ in \cite[\S4]{OO} and above theorem \ref{SLAG.OO},
are as holomorphic Lagrangian fibrations of the hyperK\"ahler rotation $\X_t^{\vee}$. Let us consider the
effect of monodromy $T\in {\rm GL}(H^2(\X_t,\ZZ))$ i.e.,
when $t$ goes around the origin.
Since $T$ preserves the isotropic fiber class $e$ of $f_t$, which is associated to the maximal degeneration $\X\to \Delta^*$,
$T$ only changes the marking hence preserves $f_t$ by the theory of harmonic integrals (\cite{Hodge}).
However, note that the induced change of marking does change the hyperK\"ahler rotation $\X_t^{\vee}$,
hence $\X_t^{\vee}$ does not form a family. Indeed, as $t$ goes around $0$, then $\X_t^{\vee}$ will be twisted i.e., have different complex structures,
while preserving their Jacobian fibrations.
\end{Rem}

\begin{Rem}
In the context of Theorem \ref{SLAG.OO}, note that each family
$(\X'|_{\Delta^*},\mathcal{L}'|_{\Delta^*})\to \Delta^*$ may have several description
as a quotient of K-trivial surfaces by finite group $H$, as
$(\X|_{\Delta^*},\mathcal{L}|_{\Delta^*})\to
(\X'|_{\Delta^*},\mathcal{L}'|_{\Delta^*})$. Indeed, for some
$(\X'|_{\Delta^*},\mathcal{L}'|_{\Delta^*})$, one of covering family
$(\X|_{\Delta^*},\mathcal{L}|_{\Delta^*})$ is a family of polarized abelian surfaces while another is
a family of polarized K3 surfaces.

Also note that if the fibers $\mathcal{X}_t$ are K3 surfaces, so are $\mathcal{X}'_t$.
However, if the fibers $\mathcal{X}_t$ are abelian surfaces, $\mathcal{X}'_t$ becomes abelian surfaces if $H$ only consists of translations,
while $X'_t$ becomes K3 surfaces otherwise.
See e.g., \cite{Fujiki}, \cite[p. 329]{Huy} and the
references therein, for more details and examples.
\end{Rem}


\subsection{Higher dimensional abelian varieties case}\label{ppav}

In this subsection, we prove Conjecture \ref{metricSYZ} for
maximally degenerating
abelian varieties in any dimension.
The main tool is the following explicit description of the degenerations,
which itself may be of interest.
It is a complex analytic version of
the well-known non-archimedean
uniformization of
abelian variety over a
non-archimedean
valued field
by
Mumford \cite{Mum72} and
Faltings-Chai \cite{FC}.
Although \cite{Od} included the
statement with a rough sketchy proof,
since we could not find precise discussion in the literature, we include it here
of somewhat more global statement,
with a
more topological or analytic proof.

\begin{Lem}[General description of maximal degenerating abelian varieties cf., \cite{FC}, \cite{Od}]\label{FC}
Take any maximally degenerating
family
of polarized abelian varieties of
any dimension $g$,
which we denote again as
$(\X|_{\Delta^*},\mathcal{L}|_{\Delta^*})$
over $\Delta^*$.
Suppose that the polarization is of type $(e_1,\cdots,e_g)$
where $e_i (1\le i\le g)$ are all positive integers which satisfy
$e_i\mid e_{i+1}$ for any $i$.
(For instance, $e_i$ are all $1$ for principal polarization.)
Then, these families are characterized explicitly as
\begin{align}\label{mav}
((\C^*)^g \times \Delta^*)
/\mathbb{Z}^g\to \Delta^*,
\end{align}
where $\mathbb{Z}^g \ni {}^t  (m_1,\cdots,m_g)$
acts on
$(\C^*)^g \times \Delta^*\ni
 (z_1,\cdots,z_g,t)$
 by
 $$(z_1,\cdots,z_g,t)\mapsto
 \left(
  (z_1,\cdots,z_g)\cdot {}^t
 \left( \prod_{1\le j\le g}p_{i,j}(t)^{m_j}\right)
 _{1\le i\le g},t \right)
 $$
for some meromorphic functions $p_{i,j}(t)\in \C((t))$.
 Here, meromorphic functions $q_{i,j}(t):=p_{i,j}(t)^{e_i}\in \C((t))$
 satisfies the following:
 \begin{enumerate}
     \item possible pole only at $t=0$.
     \item each $q_{i,j}(t)$ converges at
     $t\in \Delta^*$.
     \item $(q_{i,j}(t))_{1\le i,j\le g}$ is a symmetric matrix.
 \end{enumerate}
 From here, we
 often describe such family \eqref{mav} simply as either
 $$\bigsqcup_{t} (\mathbb{C}^*)^g/\langle p_{i,j}(t) \rangle \ \textrm{or} \
 \left( (\C^*)^g \times \Delta^*\right)
 /\langle p_{i,j}(t) \rangle.$$
\end{Lem}

\begin{proof}
Note the polarized family is induced from a map
$\tilde{\varphi}\colon \mathbb{H}\to \mathbb{H}_g$
where $\mathbb{H}(=\mathbb{H}_1)$ is the upper half plane, $\mathbb{H}_g$ is the
(higher dimensional)
Siegel upper half space of dimension $\frac{g(g+1)}{2}$,
which descends to $\varphi\colon \Delta^*\to {\rm Sp}(E,\mathbb{Z})\backslash \mathbb{H}_g$.
Here, we set
$$E:=
\begin{pmatrix}
e_1 & 0 & \cdots & 0 \\
0 & e_2 & \cdots & 0 \\
0 & 0     & \ddots & 0 \\
0 & 0     & \cdots & e_g
\end{pmatrix},
\quad
\tilde{E}:=
\begin{pmatrix}
0 & E \\
-E & 0
\end{pmatrix},$$
and
$$
{\rm Sp}(E,\ZZ)
:=\{h\in {\rm GL}(2g,\ZZ)\mid h\tilde{E} {}^{t}h=\tilde{E}\}.$$
We denote ${\rm Sp}(E,\ZZ)$
also as $\Gamma$.
Consider the (rational)
$0$-cusp $F$ of $\mathbb{H}_g$ in the Satake-Baily-Borel compactification
(before dividing by $\Gamma$).
Then the monodromy lies in the unipotent radical (cf., \cite[Chapter III]{AMRT})
\begin{align*}
U(F)_{\mathbb{Z}}&=
\biggl\{\begin{pmatrix}
I_g &  B \\
0 & I_g
\end{pmatrix}\mid B\in {\rm GL}(g,\R) \biggr\} \cap  \Gamma \\
&=
\biggl\{\begin{pmatrix}
I_g &  B \\
0 & I_g
\end{pmatrix}\mid B\in {\rm GL}(g,\ZZ),
BE = {}^t(BE) \biggr\}.
\end{align*}
Therefore, if we denote the fiber $\mathcal{X}_t (t\in \mathbb{H} {\text{ or }} \Delta^*)$ as
$\mathbb{C}^g/
\begin{pmatrix}
E \\
\Lambda_t
\end{pmatrix}\mathbb{Z}^{2g}$
with $\Lambda_t \in \mathbb{H}_g$,
for each $i\in \{1,\cdots,g\}$,
$\Lambda_t (0,\cdots,0,\overbrace{1}^{i-\text{th}},0,\cdots,0)$ is invariant modulo
$E \mathbb{Z}^{g}.$
Hence, we can write $\X^*:=\X |_{\Delta^*}$ as
$\sqcup_t ((\mathbb{C}^*)^g)/\langle p_{i,j}(t) \rangle$ for some
holomorphic functions $p_{i,j}$s over $\Delta^*$.
Here, consider a finite group action
$\mu_{e_1}\times \cdots \times \mu_{e_g}:\X^*\to \X^*$
defined by $(z_1,\dots, z_g,t)\mapsto (z_1^{e_1},\dots , z_g^{e_g}, t)$.
Then, note that the natural quotient of $\mathcal{X}^*$ by
$\mu_{e_1}\times \cdots \times \mu_{e_g}$ is a family of
principally polarized abelian varieties over $\Delta^*$,
which can be also written as
$\sqcup_t ((\mathbb{C}^*)^g)/\langle q_{i,j}(t) \rangle$ for some
holomorphic functions $q_{i,j}$s over $\Delta^*$. Here,
since
$(\frac{{\rm Im}(\log q_{i,j}(t))}{2\pi \sqrt{-1}})_{1\le i,j\le g}$
lie inside $\mathbb{H}_g$ for any $t\in \Delta^*$,
$q_{i,j}=q_{j,i}$ and $p_{i,j}(t)^{e_i}=q_{i,j}(t)$.
Again, since $(\frac{{\rm Im}(\log q_{i,j}(t))}{2\pi \sqrt{-1}})_{1\le i,j\le g}$
lie inside $\mathbb{H}_g$ for any $t\in \Delta^*$,
it follows that all
$q_{i,j}(t)$ and
$p_{i,j}(t)$ are meromorphic.
\end{proof}

Below we give additive description as well and
fix notation.

\subsubsection{Notation}\label{degav.notation}
If we describe the above abelian varieties
in Lemma \ref{FC}
in additive manner,
\begin{align}
\mathcal{X}_t   &
= \C^g/
	\begin{pmatrix}
	e_1 & 0 & 0 &\cdots & 0 &                \\
	0 & e_2 & 0 & \cdots & 0 &                \\
	0 & 0 & e_3 & \cdots & 0 &      \dfrac{ e_i\log p_{i,j}(t)}{2\pi\sqrt{-1}}\\
	   &    &    & \ddots &    &                \\
	0 & 0 & 0 & \cdots & e_g &                \\
	\end{pmatrix}\mathbb{Z}^{2g}
	\label{additive1}
	\\
	&
	\label{2}
	=((\C^*)^g /\langle p_{i,j}(t) \rangle_{i,j})
\end{align}
where the first identification is through the
exponential map
$$z_i\mapsto e^{\frac{2\pi \sqrt{-1}z_i}{e_i}}=:Z_i (i=1,\cdots,g).$$
If we set
\begin{align}
E&:={\rm diag}(e_1,\cdots,e_g)\\
 &= \begin{pmatrix}
	e_1 & 0 & 0 &\cdots & 0 &                \\
	0 & e_2 & 0 & \cdots & 0 &                \\
	0 & 0 & e_3 & \cdots & 0 &     \\
	   &    &    & \ddots &    &                \\
	0 & 0 & 0 & \cdots & e_g &                \\
	\end{pmatrix}
\end{align}
\begin{align}\label{Omega}
\Omega (t):=\left( \frac{1}{2\pi \sqrt{-1}} \log p_{i,j}(t)\right)_{i,j},
\end{align}
and
\begin{align}
\Omega' (t):=\left( \frac{e_i}{2\pi \sqrt{-1}} \log p_{i,j}(t)\right)_{i,j},
\end{align}
then the dividing lattice of the above \eqref{additive1}
splits as
$E\mathbb{Z}^g\oplus \Omega'(t)\mathbb{Z}^g$ i.e.,
\begin{align}\label{Lambdat}
\Lambda_t&:=\begin{pmatrix}
	e_1 & 0 & 0 &\cdots & 0 &                \\
	0 & e_2 & 0 & \cdots & 0 &                \\
	0 & 0 & e_3 & \cdots & 0 &      \dfrac{e_i \log p_{i,j}(t)}{2\pi\sqrt{-1}}\\
	   &    &    & \ddots &    &                \\
	0 & 0 & 0 & \cdots & e_g &                \\
	\end{pmatrix}\mathbb{Z}^{2g}
	\\
	&=E\mathbb{Z}^g \oplus \Omega'(t)\mathbb{Z}^g.
\end{align}
Note that, $\Omega'(t)=E\Omega(t)$ is symmetric, and $\Omega(t)$ (resp.,
$\Omega'(t)$)
is determined up to $\ZZ^g$
(resp.,
$E\ZZ^g$).

 We continue to use this notation for the coordinates $z_i$ and $Z_i$,
 and set $x_i:={\rm Re}(z_i), y_i:={\rm Im}(z_i)$.
 We write
 \begin{align}\label{alpha.beta}
 \frac{\log p_{i,j}(t)}{2\pi \sqrt{-1}}=\alpha_{i,j}(t)+\sqrt{-1} \beta_{i,j}(t),
 \end{align}
 with an ambiguity of $\alpha_{i,j}(t)$ modulo $\ZZ$.
 Indeed, if $t$ goes around the origin $0$, the monodromy effect is nontrivial only on $\alpha_{i,j}$ but trivial on $\beta_{i,j}(t)$.
 Similarly, for $q_{i,j}(t)=p_{i,j}(t)^{e_i}$, we write
 \begin{align}\label{alpha.beta2}
 \frac{\log q_{i,j}(t)}{2\pi \sqrt{-1}}=\alpha'_{i,j}(t)+\sqrt{-1} \beta'_{i,j}(t).
 \end{align}
In particular, $(\beta'_{i,j}(t))={\rm Im} \Omega' (t)$.

 Below, we prepare a basic lemma
 on the symmetry of maximally degenerating
 family, for the next theorem
 \ref{SLAG.ppav} \eqref{fin.sym.thm}
 which discusses finite quotients of
 abelian varieties.
 We explore more details in
 later section \ref{sec.sym}.

 \begin{Lem}\label{fin.sym.lem}
 For the family of above Lemma \ref{FC},
 i.e., $$\X^*=\sqcup_{t\in \Delta^*}((\C^*)^g /\langle p_{i,j}(t) \rangle_{i,j})_t \to \Delta^*_t,$$
 we continue to use the
 notation in \eqref{degav.notation}.

 Suppose there is an action of a group $H$  holomorphically on $\X^*$ preserving the fibers
 $\X_t$
 and $c_1(\mathcal{L}_t$).
 We do not assume it preserves the $0$-section.
 For each $t\in \Delta^*$, recall that
 \begin{align}
 \Lambda_t&=E \mathbb{Z}^g \oplus \Omega'(t)\mathbb{Z}^g
 \label{split}
 \end{align}
 as a sublattice of $\C^g$.
 Then, for any $h\in H$,
 the induced $h_*\colon H_1(\mathcal{X}_t,\mathbb{Z})
 \to H_1(\mathcal{X}_t,\mathbb{Z})$ preserves the
 first direct summand $\mathbb{Z}^g$ of the above
 \eqref{split}. Restriction of $h_*$ to it is
 denoted as $l_h\in {\rm GL}(g, \ZZ)$.
 Further,
 $$
 \begin{array}{cccc}
 l\colon &H                     & \longrightarrow &          {\rm GL}(g,\ZZ)            \\
 & \rotatebox{90}{$\in$} &                 & \rotatebox{90}{$\in$} \\[-1pt]
 &h                     & \longmapsto     & l_h
 \end{array},
 $$
 is a group homomorphism.
 \end{Lem}

 \begin{proof}
 The $H$-action on $\mathcal{X}_t$ induces
 a complex linear transformation
 $f_h\in {\rm GL}(g,\C)$, i.e., at the level of universal covers of $\X_t$s,
 and what we want to show is that
 the image lies inside ${\rm GL}(g,\ZZ)$.
 If we take only real $t\in \R_{>0}$, we can take a (continuous) branch of $\beta_{i,j}(t)$ to
 avoid possible ambiguity due to
 the monodromy.
 Then, each $\Lambda_t$ for $t\in \R_{>0}$ can be canonically identified with $\ZZ^{2g}$. Take $e_1,\cdots, e_g, f_1,\cdots,f_g$ as
 the standard basis of $\Lambda_t=\ZZ^{2g}$, i.e., set
 $$e_i:=(0,\cdots,0,\overbrace{1}^{i-\text{th}},0,\cdots,0)$$
 and
 $$f_i:=(0,\cdots,0,\overbrace{1}^{(i+g)-\text{th}},0,\cdots,0)$$
 for each $i \in \{1,\cdots,g\}$.

 Now we prove the assertion by contradiction. Assume the contrary i.e., there is some $h\in H$ and $1\le i\le g$ such that
 $l_h(e_i)$ does not sit inside
 $\sum_{1\le i\le g}\mathbb{R}e_i$.
 Since ${\rm GL}(\Lambda_t)={\rm GL}(2g,\ZZ)$ is a discrete group,
 the $l_h$ is induced by  a particular isomorphism $g_h \colon \Lambda_t\to \Lambda_t$ for {\it any} $t$.
 That is, for any $h$ and $t\in \Delta^*$, $l_h(\Lambda_t)=\Lambda_t$ so that $g_h$ is obtained as a restriction of $l_h$.
 From the assumption that imaginary part of
 $g_h(e_i)=l_h(e_i)$ is not the zero vector,
 if we write
 $$g_h(e_i)=\sum_{1\le j\le g}a_j e_j+ \sum_{1\le j\le g}b_j f_j.$$
 where all $a_j, b_j$s are integers,
 one can assume $b_j\neq 0$ for some $j$.
 Note $$(\beta'_{i,j}(t))_{1\le i,j\le g}\sim \dfrac{-\log|t|}{2\pi}(B'_{i,j})_{1\le i,j\le g},$$
 with positive definite symmetric matrix $(B'_{i,j})$, a part of Faltings-Chai degeneration data.
 Here, $\sim$ means
 the ratio converges to $1$ when $t\to 0$.
 Also, the flat metric's Gram matrix is,
 if we set
 $X(t)=(\alpha'_{i,j}(t))_{i,j}$, $Y(t)=(\beta'_{i,j}(t))_{i,j}$,
 then
 \begin{align}
 &\begin{pmatrix}
 E &   0\\
 X(t) & Y(t)\\
 \end{pmatrix}
 \begin{pmatrix}
 Y(t)^{-1} &   0\\
 0 & Y(t)^{-1}\\
 \end{pmatrix}
 \begin{pmatrix}
 E &   X(t)\\
 0 & Y(t)\\
 \end{pmatrix}\\
 &=\begin{pmatrix}
 EY(t)^{-1}E    &   EY(t)^{-1}X(t)\\
 X(t)Y(t)^{-1}E & X(t)Y(t)^{-1}X(t)+Y(t)\\
 \end{pmatrix}
 \end{align}
 (c.f., e.g. \cite[(3) during the proof of Theorem 2.1]{Od}) so that in particular
 $|e_i| \to 0$ while $|f_h(e_i)|\to \infty$ for $t\to 0$. Here, $| \cdot |$ denotes the length
 with respect to the canonical flat K\"ahler metric $g_{\rm KE}$ on $\X_t$.
 This contradicts with the fact that
 $H$-action on $\mathcal{X}_t$ is an isometry for all $t\neq 0$. Hence the assertion is proven by contradiction.
 \end{proof}

Now, we are ready to show the existence and description of
special Lagrangian fibrations for maximally degenerating abelian varieties.

 \begin{Thm}\label{SLAG.ppav}
 For any maximally degenerating
 family $(\X|_{\Delta^*},\mathcal{L}|_{\Delta^*})/\Delta^*$
 of polarized abelian varieties of dimension $g$,
 we again denote by $\omega_t$  the K\"ahler form
 of the flat metric $g_{\rm KE}(\X_t)$
 with $[\omega_t]=c_1(\mathcal{L}_t)$.
 We also take a certain non-zero element
 $(0\neq ) \Omega_t\in H^0(\X_t,\omega_{\X_t}=\wedge^g \Omega_{\X_t})$.
 (See the proof for the actual choice of
 $\Omega_t$.)
 Recall that
 $\X_t$ can be described as $\C^g/\Lambda_t$ for the lattice $\Lambda_t$ of the form \eqref{Lambdat} of Lemma \ref{fin.sym.lem}.

 Then, the following holds:

 \begin{enumerate}
 \item \label{SLAG.ppav.existence}
 For each $t \in \Delta^*$,
 special Lagrangian fibration $f_t\colon \X_t \to \mathcal{B}_t$
 with respect to $\omega_t$ and $\Omega_t$ exists
 ($g=2$ case is proven in Theorem \ref{SLAG.OO}).
 $\mathcal{B}_t$ is a $g$-dimensional real torus described as
 $\R^g/ (\beta_{i,j}(t))_{1\le i,j\le g}\ZZ^g$
  under the notation \eqref{alpha.beta}.
 Furthermore, they fit into a family i.e.,
 there is a continous map $f\colon \X|_{\Delta^*}\to \cup_{t\neq 0}\mathcal{B}_t$
 for a certain extended topology on $\cup_{t\neq 0}\mathcal{B}_t$, which restricts to $f_t$ for each $t\neq 0$.
 \item \label{A.ppav}
 The tropical affine structure $\nabla_A(t)$ on $\mathcal{B}_t=\R^g/ (\beta_{i,j}(t))_{1\le i,j\le g}\ZZ^g$ is descended from
 the integral affine structure on $\R^g$
 whose integral points are
 $ (\gamma_{i,j}(t))_{1\le i,j\le g}\mathbb{Z}^g$,
 where $$(\gamma_{i,j}(t))_{i,j}:= E^{-1} \cdot (\beta'_{i,j}(t))_{i,j} \cdot E^{-1}$$ under the notation \eqref{alpha.beta2}.
 \item \label{B.ppav}
 The tropical affine structure $\nabla_B(t)$ on $\mathcal{B}_t=\R^g/ (\beta_{i,j}(t))_{1\le i,j\le g}\ZZ^g$ is descended from
 the integral affine structure on $\R^g$ whose integral points are
 $\frac{-\log |t|}{2\pi}\ZZ^g$.
 \item \label{fin.sym.thm}
 Suppose there is a fiber-preserving holomorphic action of a group $H$ on $(\mathcal{X},c_1(\mathcal{L}))$
 (as in Lemma \ref{fin.sym.lem}).
 Then for each $t\neq 0$,
 the $H$-action on $\mathcal{X}_t$
 descends to that on
 $\mathcal{B}_t$ which is unified into a
 continuous $H$-action on
 $\cup_{t\neq 0}\mathcal{B}_t$ with respect to
 which $\X^*:=\X|_{\Delta^*}\to \cup_{t\neq 0}\mathcal{B}_t$
 is $H$-equivariant.

 Further, if the group $H$ also fixes the isomorphic class of  $\mathcal{L}_t$ for each $t\neq 0$  (not only their $c_1$),
 then $H$ must be a finite group. In that case,
 we can take the quotient of the map $f$ by $H$.
 We denote the obtained map as
 $f/H\colon \X^*/H\to \cup_{t\neq 0}(\mathcal{B}_t/H)$.

 \item \label{fin.sym.thm2}
 Under the situation of above
 \eqref{fin.sym.thm}, suppose further that
 the $H$-action on $\mathcal{X}^*$
 preserves relative holomorphic $n$-forms i.e.,
 acts trivially on  $\omega_{\mathcal{X}^*/\Delta^*}=\mathcal{O}(K_{\mathcal{X}^*/\Delta^*})$.
 Then, for each $t\neq 0$, $f/H$ is restricted to a
 special  Lagrangian fibration with respect to the descent of
 $\omega_t$ and $\Omega_t$. Further, they are of the form
 $g_t\colon \X_t/H \to \mathcal{B}_t/H$, where the additional structures on $\mathcal{B}_t/H$ is those descended from the ones described above
 \eqref{A.ppav}, \eqref{B.ppav}.
 \end{enumerate}
 \end{Thm}

 \begin{Rem}
 A weaker version of the above statements \eqref{SLAG.ppav.existence} i.e.,
 existence of special Lagrangian fibrations in {\it large open subset of } $\mathcal{X}_t$
 are essentially proven also as the simple combination of \cite[Theorem 1.3]{YangLi} and \cite{Liu}.
 \eqref{SLAG.ppav.existence} of
 the above result refines
 it, even in an explicit
 way.
 \end{Rem}

 \begin{proof}
 We set $B'_{i,j}:={\rm val}_t(q_{i,j}(t))$.
  Recall that \cite{Od, OO, Goto} repeatedly proved that the
 Gromov-Hausdorff limit of $\X_t$ is identified with the $g$-dimensional torus with the
 Gram matrix $(B'_{i,j})_{1\le i,j\le g}$.
 Indeed,  the standard description of Riemann forms tells us that
 \begin{align}\label{Kahler.ppav}
 \omega_t=\frac{\sqrt{-1}}{2}\sum_{1\le i,j\le g}\gamma_{i,j}(t)dz_i\wedge d\overline{z_j},
 \end{align}
 where $(\gamma_{i,j}(t))_{i,j}=E\cdot (\beta'_{i,j}(t))_{i,j}^{-1}\cdot E$ for the
  the basis $\{ z_i\}$ of $\C^g$ such that the period matrix is of the form $(I\ \Omega(t))$ under the notation \eqref{Omega}.

 From below, we use $$\Omega_{t}=
 \frac{-2\pi}{\log |t|} dz_1\wedge \cdots \wedge dz_g \in H^0(\mathcal{X}_t,K_{\mathcal{X}_t}).$$
 This choice is carefully made as multiplication by constant does change the associated
 special Lagrangian fibrations.

 \vspace{4mm}
 Now we prove \eqref{SLAG.ppav.existence}.
 We use Lemma \ref{FC} to set $\mathcal{X}_t=((\C^*)^g /\langle p_{i,j}(t) \rangle_{i,j})$, and
 use the description and the notations \eqref{2}, \eqref{alpha.beta} and \eqref{Kahler.ppav} above. Then,
 for each fiber $\mathcal{X}_t$,
 maps defined by $$
 \begin{array}{cccc}
 \tilde{f}_t &\colon     \C^g                 & \longrightarrow & \R^g              \\
 &\rotatebox{90}{$\in$} &                 & \rotatebox{90}{$\in$} \\
 &       (z_1,\cdots,z_g)             & \longmapsto     &(y_1,\cdots,y_g),
 \end{array}
 $$
 where $y_i = {\rm Im} z_i$,
 descend to
 $$f_t\colon \mathcal{X}_t\to \R^g/(\beta_{i,j})_{1\le i,j\le g}\ZZ^g = \R^g / {\rm Im} \Omega (t) \ZZ^g$$
under the notations \eqref{Omega} and  \eqref{alpha.beta}.

 We denotes an arbitrary fiber of $f_t$ as $F(\cong \R^g / \ZZ^g)$ whose coordinates are defined by $x_i={\rm Re} z_i$.
 From \eqref{Kahler.ppav}, $\omega_t|_F=0$.
 Also, it is easy to see that
 $${\rm Im}\left( \Omega_t=\frac{-2\pi}{\log |t|}dz_1\wedge \cdots \wedge dz_g\right)\Biggr|_F=0.$$
 Hence \eqref{SLAG.ppav.existence} follows. Next, \eqref{A.ppav} and \eqref{B.ppav} follows from the definition in \cite{Hit, Gross} and
 standard calculation below. Since the fibers of $f_t$ are the integral manifolds along the ``real" directions $x_i$s,
 the affine coordinates along $\nabla_A(t)$ (resp., $\nabla_B(t)$)
 of $\frac{\partial}{\partial y_i}$
  is determined as
 \begin{align*}
 &
 \Biggl( \int_{0\le x_1\le 1}
 \iota\biggl(\frac{\partial}{\partial y_i}
 \biggr)\sum_{1\le i,j\le g}\frac{\sqrt{-1}}{2}\gamma_{i,j}(dx_i+\sqrt{-1}dy_i)\wedge (dx_j-\sqrt{-1}dy_j)  ,\\
 &   \qquad \qquad \qquad \qquad \qquad \qquad\qquad\cdots \cdots,\\
 & \int_{0\le x_{g}\le 1}
 \iota\biggl(\frac{\partial}{\partial y_i}
 \biggr)\sum_{1\le i,j\le g}\frac{\sqrt{-1}}{2}\gamma_{i,j}(dx_i+\sqrt{-1}dy_i)\wedge (dx_j-\sqrt{-1}dy_j)\Biggr) \\
 &=(\gamma_{1,i},\cdots,\gamma_{g,i})
 \end{align*}
 (resp.,
 \begin{align*}
 &
 \Biggl(\int_{\Gamma_1 }
 \iota\biggl(\frac{\partial}{\partial y_i}
 \biggr){\rm Im}\left( \frac{-2\pi}{\log |t|} (dx_1+\sqrt{-1}dy_1)\wedge \cdots \wedge (dx_g+\sqrt{-1}dy_g)\right) ,\\
 &   \qquad \qquad \qquad \qquad \qquad \qquad\qquad\cdots \cdots,\\
 & \int_{\Gamma_g }
 \iota\biggl(\frac{\partial}{\partial y_i}
 \biggr){\rm Im}\left(  \frac{-2\pi}{\log |t|} (dx_1+\sqrt{-1}dy_1)\wedge \cdots \wedge (dx_g+\sqrt{-1}dy_g)\right) \Biggr)\\
 &= \frac{-2\pi}{\log |t|} (0,\cdots,0,\overbrace{1}^{i\text{-th}},0,\cdots,0),
 \end{align*}
 where $\Gamma_i$ is the $g-1$-cycle on $f_t^{-1}(y) = F_y$ of the form
 $$[0,1]\times \cdots \times \{ x_i\} \times  \cdots \times [0,1] \subset \R^g / \ZZ^g \cong F_y).$$
Set $\{w_i\}$ (resp., $\{ \check{w_i}\}$) as
the integral affine coordinates induced by $\nabla_A(t)$ (resp., $\nabla_B(t)$).
The above forms are rephrased as
$${}^t (w_1,\dots , w_g)= \left(E\left({\rm Im} \Omega'(t)\right)^{-1}E\right)  {}^t (y_1,\dots , y_g)$$
and
$${}^t (\check{w_1},\dots , \check{w_g})= \left(\frac{-2\pi}{\log |t|}\right) {}^t (y_1,\dots , y_g).$$
Hence, by taking the inverse matrices, we can see that $\nabla_A(t)$ and $\nabla_B(t)$ give the integral affine structures
on $ \R^g/{\rm Im} \Omega (t)\ZZ^g$
defined by $E^{-1}{\rm Im} \Omega'(t)E^{-1}\ZZ^g$ and
$\frac{-\log |t|}{2\pi}\ZZ^g$, respectively.
Note that $E^{-1}{\rm Im} \Omega'(t)E^{-1}={\rm Im} \Omega(t) E^{-1}$.
In particular, the Mclean metric $g_t$ with respect to $\{y_i\}$
is of the form $$\frac{-2\pi}{\log |t|}
E({\rm Im} \Omega(t))^{-1}.$$
Indeed, by definition, the Mclean metric $g_t$ is defined as follows:
$$\check{w_i}=g_t\left(\frac{\partial}{\partial w_i},\frac{\partial}{\partial w_j}\right)w_j.
$$
See also \cite[Definition 1.1]{Gross}.
Hence,
$$\left(g_t\left(\frac{\partial}{\partial w_i},\frac{\partial}{\partial w_j}\right)\right)_{i,j}
= \frac{-2\pi}{\log |t|} \left( {\rm Im} \Omega (t)\right) E^{-1}.
$$
Since $ E \left( {\rm Im} \Omega (t)\right)^{-1}= {}^t\left(  E \left( {\rm Im} \Omega (t)\right)^{-1}\right)$,
we obtain the following:
\begin{align*}
\left(g_t\left(\frac{\partial}{\partial y_i},\frac{\partial}{\partial y_j}\right)\right)_{i,j}
&= \frac{-2\pi}{\log |t|}\cdot {}^t\left(E\left( {\rm Im} \Omega (t)\right)^{-1}\right)   {\rm Im} \Omega (t) E^{-1} E
\left( {\rm Im} \Omega (t)\right)^{-1} \\
&=\frac{-2\pi}{\log |t|} E \left( {\rm Im} \Omega (t)\right)^{-1}.
\end{align*}

 Next we show \eqref{fin.sym.thm}.
 For the former statement, i.e., to descend the $H$-action to $\mathcal{B}|_{\Delta^*}$,
 it is enough to
 show that each
 $h\in H$ sends  any fiber of $f_t$ for $t\neq 0$
 to one of its fibers.
 Note that they are written in the form
 $$\overline{\mathbb{R}^g+\sqrt{-1}(c_1,\cdots,c_g)}
 \subset \mathbb{C}/(E\ZZ^g \oplus \Omega(t)\ZZ^g).$$
 Therefore the claim follows
 Lemma \ref{fin.sym.lem} as
 the coefficients of the
 linear part of $h$-action are real numbers
 (actually integers).

 When the $H$-action preserves the isomorphic class of
 $\mathcal{L}_t$ for each $t\neq 0$,
 $H$ is regarded as a subgroup of the automorphism group
 of
 polarized abelian variety of
 $(\mathcal{X}_t,\mathcal{L}_t)$ for a fixed
 $t\neq 0$,
 which is well-known to be finite.
 Hence, $H$ itself is finite.
 Thus, combined with the former statements of
  \eqref{fin.sym.thm}, $f/H$ exists.

 The remained proof of \eqref{fin.sym.thm2} is
 immediate from the definition of
 special Lagrangian fibrations.
 \end{proof}
 \begin{Rem}
 In particular, the conjectural asymptotic formula of Ricci-flat K\"ahler forms by \cite[(11) also c.f., \S 4.2, \S 4.5]{YangLi}
 holds even globally in this abelian varieties case.
 \end{Rem}

\begin{Rem}
Here we confirm that
the special Lagrangian fibrations we constructed in Theorem \ref{SLAG.OO} and
Theorem \ref{SLAG.ppav}
are compatible. That is, for the case of principally polarized abelian surfaces which
are common to both setup, the maps $f_t$ in Theorem \ref{SLAG.ppav}
are examples of that of Theorem \ref{SLAG.OO} although the latter are not completely unique
(due to the small ambiguity of applying Siegel reduction as in \cite[\S 4, (4.11)]{OO}), to be precise.

Indeed, if we put $(y'_1,y'_2):=(\beta_{i,j})_{1\le i,j\le 2}\cdot ^{t}(y_1,y_2)$, we have
\begin{align}
{\rm Re}(\Omega_t)&=dx_1\wedge dx_2- ({\rm det}(\beta_{i,j}))^{-1}dy'_1\wedge dy'_2,\\
{\rm Im}(\Omega_t)&=dx_1\wedge (\beta'_{2,1} dy'_1+\beta'_{2,2}dy'_2)
- dx_2\wedge (\beta'_{1,1} dy'_1+\beta'_{1,2}dy'_2).
\end{align}
Topologically all the fibers of $f_t$ (of Theorem \ref{SLAG.ppav})
are identified with the $4$-torus $T$ with coordinates $x_1, x_2, y'_1, y'_2$.
In $H^2(T,\R)$, if we put
\begin{align*}
e&:=[dx_1\wedge dx_2], \\
f&:=[dy'_1\wedge dy'_2] \text{ and  }\\
v&:=[dx_1\wedge (B'_{2,1}dy'_1+B'_{2,2}dy'_2)
-dx_2\wedge (B'_{1,1}dy'_1+B'_{1,2}dy'_2)],
\end{align*}
where $(B_{i,j}:={\rm val}_t p_{i,j}(t))$,
$(B'_{i,j})_{i,j}:=(B_{i,j})^{-1}_{i,j}$ as a matrix,
the above gives
\begin{align}
\label{7} \log|t|\cdot {\rm Re}(\Omega_t)&=\log|t| e-O\biggl(\frac{1}{\log|t|}\biggr)f \\
\label{8} \log|t|\cdot {\rm Im}(\Omega_t)&=(v+o(1)),
\end{align}
for $t\to 0$. As these \eqref{7}, \eqref{8} fit to the reduction condition \cite[p.47, (4.11) and l.10-12]{OO}
(set $N_i:=\log|t|, \epsilon_i:=\frac{1}{N_i}$ in the notation of {\it loc.cit}),
it means the desired compatibility with the construction of $f_t$ in Theorem \ref{SLAG.OO}.
\end{Rem}



\section{Hybrid SYZ fibration and Kontsevich-Soibelman conjecture}\label{hyb.sec}

As briefly discussed in the introduction, this section discusses to glue two kinds of fibrations in originally very different natures:
\begin{itemize}
\item \label{C.SYZ}
special Lagrangian fibrations (SYZ fibrations) and
\item \label{NA.SYZ}
non-archimedean SYZ fibration,
\end{itemize}
which is a sort of enhanced answer to \cite[Conjecture 3]{KS}. Indeed, we use it to show
the conjecture for certain finite quotients of
abelian varieties, generalizing \cite[\S 5]{Goto}.
For the gluing,
we use the recent technology of hybrid norms originated in \cite{Ber.Manin} and
re-explored in \cite[\S 2]{BJ}, \cite[Appendix]{Od}.
It is also closely related to
earlier
Morgan-Shalen's partial compactification
technique \cite{MS}.

More precisely speaking, \cite{BJ} shows the following by two different constructions i.e.,
as the projective limit of Morgan-Shalen type extension and as a variant of Berkovich analytification.
\begin{Lem}[{\cite[\S 4 and Appendix]{BJ}}]\label{hyb}
For each smooth projective family $\pi^*\colon \X^*\to \Delta^*$ and associated
smooth projective variety $X$ over $\C((t))^{\rm mero}$,
there is a topological space $\pi^{\rm hyb}\colon
\X^{*,{\rm hyb}}\to \Delta$ such that
\begin{enumerate}
\item $\pi^{\rm hyb}|_{\Delta^*}=\pi^*\colon \X^*\to \Delta^*$ with the complex analytic topology,
\item $(\pi^{\rm hyb})^{-1}(0)=X^{\rm an}$, where $X^{\rm an}$ means the Berkovich analytification of $X/\C((t))^{\rm mero}$.
\end{enumerate}
\end{Lem}

The following is one of our main theorems.
We have not found literature
which conjectured the statements.

\begin{Thm}[Hybrid SYZ fibration]\label{Hybrid SYZ}
For any maximally degenerating
family $(\X|_{\Delta^*},\mathcal{L}|_{\Delta^*})/\Delta^*$
of polarized abelian varieties of dimension $g$
we use the same notation as Theorem \ref{SLAG.ppav}. Then, the following hold.

\begin{enumerate}
\item \label{AV.hybrid}
There is a family of
tropical affine manifolds (real tori)
i.e.,
a topological space
$\mathcal{B}:=\sqcup_{t\in \Delta}\mathcal{B}_t$
with a natural continuous map to
$\Delta=\{t\in \mathbb{C}\mid |t|<1\}$
and continuous family of tropical affine
structures on $\mathcal{B}_t$ such that
for $t\neq 0$, they coincide with
$(\mathcal{B}_t,\nabla_B(t))$ of
Theorem \ref{SLAG.ppav}.

Further, it comes with
a continuous proper map
\begin{align}
f^{\rm hyb}\colon \X^{*, \rm hyb}\to \mathcal{B},
\end{align}
such that
\begin{enumerate}
\item for $t\neq 0$, $f^{\rm hyb}|_{\X_t}=f_t$ i.e., coincides with
the special Lagrangian fibrations
in Theorem \ref{SLAG.ppav}.
\item $f^{\rm hyb}|_{t=0}=f^{\rm hyb}|_{X^{\rm an}}$ is the non-archimedean SYZ fibration appeared in \cite{KS, NXY, Goto}.
(In \cite[\S 5]{Goto}, it was denoted by $\rho_{\mathcal{P}}$.)
\end{enumerate}
We call this map
$f^{\rm hyb}$ a {\it hybrid SYZ fibration}.

\item \label{finite.hybrid}
If a  group $H$  acts holomorphically in a fiber-preserving manner on
 $\X|_{\Delta^*}$, which fixes $c_1(\mathcal{L}_t)$
 for $t\neq 0$,
 it induces continuous actions on
 both $\X^{*, \rm hyb}$ and $\mathcal{B}$.
 (See Theorem \ref{hybrid equivariant} for more details.)
 Furthermore,
$f^{\rm hyb}$ is $H$-equivariant with respect to the induced actions of $H$.
If the group $H$ also fixes $\mathcal{L}_t$,
 then $H$ must be a finite group and it descends to the quotient $\X^{*, \rm hyb}/H \to \mathcal{B}/H$,
which we denote as
$f^{\rm hyb}/H$.

\item \label{quotientasSYZ}
Under the setup of above
\eqref{finite.hybrid},
we further assume that
$H$-action acts trivially on
$\omega_{\mathcal{X}^*/\Delta^*}$
so that $\mathcal{X}^*/H\to
\Delta^*$ is again relatively
K-trivial, and the homomorphism $\iota : H \to {\rm GL}(g,\ZZ)$  appeared in Lemma \ref{fin.sym.lem} is injective.
Then,
$f^{\rm hyb}/H\colon
\X^{*, \rm hyb}/H \to \mathcal{B}/H$ is again fiberwise
special Lagrangian fibrations for $t\neq 0$ and
non-archimedean SYZ fibration for $t=0$
in a certain generalized sense than
\cite{KS, NXY}. (We clarify the precise
meaning of the generalized version
of non-archimedean SYZ fibration over
$t=0$ during the
proof. )
\end{enumerate}

\end{Thm}

\begin{proof}
First we prove \eqref{AV.hybrid}.
Recall from Lemma \ref{hyb} and \cite{BJ} that
$\X^{*,\rm hyb}$ is the projective limit of Morgan-Shalen type spaces
$\X^{*,\rm hyb}(\X):= \X^*\sqcup \mathcal{D}(\X_0)$ with the hybrid topology of \cite[\S 2]{BJ},
where $\X^*\subset \X\to \Delta$ runs over SNC models and $\mathcal{D}(\X_0)$ denotes the dual intersection complex of the  central fiber
of $\mathcal{X}$.

Here we take the Mumford construction by \cite{Mum72, Kun} (also cf., \cite{Goto, MN}) and
construct a smooth minimal
model $\X^*\subset \X$
which is projective over $\Delta$ and
the central fiber $\X_0$ is simple normal
crossing. Then,
recall from \cite[\S 2]{BJ}
that the
Morgan-Shalen type partial compactification
$(\X^*\subset) \X^{*,\rm hyb}(\X)$
satisfies
\begin{itemize}
    \item $\pi^{\rm hyb}$ extends to a continuous proper map
    $\overline{\pi^{\rm hyb}}\colon \X^{*,\rm hyb}(\X)\to \Delta$,
    \item $(\overline{\pi^{\rm hyb}})^{-1}(0)$
    is the dual complex of $\X_0$.
\end{itemize}
See \cite[\S 2]{BJ} (also cf., original
\cite{MS}) for the details including
the description of the topology.
Below, we will show that there is a map
$f(\X)\colon \X^{*,\rm hyb}(\X)\to \mathcal{B}$ such that
the composite of $f(\X)$ together with the natural retraction $\X^{*,\rm hyb}\to \X^{*,\rm hyb}(\X)$ (\cite{Ber99},
see
also \cite{KS, BJ})
gives the desired $f^{\rm hyb}$.

We take an arbitrary sequence
$$P_k=(t_k, {(Z_i)_k}) \in \Delta^*\times (\C^*)^g \hspace{2mm}
(k=1,2,\cdots)$$
which satisfies that for each $i\in \{1,\cdots,g\}$, there is a real constant $c_i$ such that
\begin{align}
\label{ass}
\dfrac{\log|(Z_i)_k|}{\log|t_k|}\to c_i
\end{align}
for $k\to \infty$.
This condition means that, the points $P_k \hspace{2mm} (k=1,2,\cdots)$ converge to $(c_1,\cdots,c_g)\in \mathbb{R}^g$
modulo
$(B_{i,j})\ZZ^g$
in the Morgan-Shalen type partial compactification
$\X^{*,\rm hyb}(\X)$.
By the definition of
the topology put on the Morgan-Shalen
type compactification
(cf., \cite{MS}, \cite[\S 2]{BJ}),
to prove the desired assertion, it is enough to show that
$f_{t_k}(P_k)$ converges to the same
$\overline{(c_1,\cdots,c_g)}
\in \mathcal{B}_0\subset \mathcal{B}$ for $k\to \infty$.

So we consider the sequence $\{f_{t_k}(P_k)\}\in \mathcal{B}$.
By Theorem \ref{SLAG.ppav} \eqref{A.ppav} and \eqref{B.ppav},
it is represented by
\begin{align}
{\rm Im}\biggl(\frac{\log(Z_i)_k}{2\pi \sqrt{-1}}\biggr) \text{ mod. }(\beta_{i,j}(t_k))_{1\le i,j\le g}\ZZ^g,
\end{align}
hence as a point in $\mathcal{B}_{t_k}\subset \mathcal{B}$,
\begin{align}\label{wts}
f_{t_k}(P_k)=\biggl(\frac{\log(Z_i)_k}{\log|t_k|}\biggr)
\text{ mod. }\biggl(\frac{2\pi \beta_{i,j}(t_k)}{-\log|t_k|}\biggr)_{1\le i,j\le g}\ZZ^g.
\end{align}
Since ${\rm Im}\frac{t}{\sqrt{-1}}=-\log|t|$ in general, applying to
\eqref{wts}, $f_{t_k}(P_k)$ converge to $\overline{(c_1,\cdots,c_g)}\in \mathcal{B}_0\subset \mathcal{B}$ for $k\to \infty$
by the assumption \eqref{ass} on the sequence $P_k$.
To finish the proof of
\eqref{AV.hybrid}, we confirm that
the obtained $f^{\rm hyb}$ is proper.
It follows from the fact that
$\X^{*,\rm hyb}(\X)\to \Delta$
for each SNC model $\X$ is always
a proper map,
together with the description of
$\X^{*,\rm hyb}$ as the projective
limit of $\X^{*,\rm hyb}(\X)$ for
a certain projective system of
the model $\X$s
(\cite[Appendix A, Theorem 10]{KS}, \cite[4.12]{BJ})
thanks to the Tychonoff's  theorem.

\vspace{2mm}
Next we prove \eqref{finite.hybrid}.
It
follows from the functoriality of the construction of hybrid analytification (\cite{Ber.Manin})
that
the $H$-action induces
a natural continuous
$H$-action on the whole $\X^{*, \rm hyb}$.
Theorem \ref{SLAG.ppav} \eqref{fin.sym.thm}
states
the existence of continuous
$H$-action
$\mathcal{B}|_{t\neq 0}$
with which
$f^{\rm hyb}|_{t\neq 0}$
is $H$-equivariant, which
we want to extend to the
whole base $\mathcal{B}$.
Recall there is a
$H$-equivariant SNC minimal  model of $\mathcal{X}^*$
by \cite[3.5, esp. (v)]{Kun},
applied over the DVR of
holomorphic germs  $\mathcal{O}_{\mathbb{C},0}^{\rm hol}$
at $0\in \mathbb{C}$.
Hence, there is a natural $H$-action
on $\mathcal{B}_{t=0}$
so that
$f^{\rm hyb}|_{t=0}$ is
$H$-equivariant.
Hence, summing up above,
we have a $H$-action on $\mathcal{B}$ with which
$f^{\rm hyb}$ is
$H$-equivariant.
What remains to
complete the proof of
\eqref{finite.hybrid}
is the continuity of the
$H$-action on $\mathcal{B}$.

Suppose the contrary.
Then there is a sequence
$x_i \hspace{2mm} (i=1,2,\cdots)\in \mathcal{B}$ and $h\in H$
such that while
$\lim_{i\to \infty}x_i$ exists in $\mathcal{B}_0$, but
\begin{align}\label{notcoincide}
\lim_{i\to \infty}h\cdot x_i
\neq h\cdot \left(\lim_{i\to \infty}x_i\right).
\end{align}
We can
lift each $x_i$ to
$\tilde{x_i}\in
\mathcal{X}^{\rm  hyb,*}(\mathcal{X})$
whose image in $\mathcal{B}$
is $x_i$. From the
properness of $f^{\rm hyb}$
proven in \eqref{AV.hybrid}
above
and locally compactness of
$\mathcal{B}$,
we can and do assume that
that
$\lim_{i\to \infty}
\tilde{x_i}$ exists.
Then,
$\lim_{i\to \infty}
h\cdot \tilde{x_i}$
maps down to
$h\cdot
(\lim_{i\to \infty}x_i)$
which contradicts with
\eqref{notcoincide}.
Therefore, we can take the
finite quotient
$f^{\rm hyb}/H\colon
\X^{*, \rm hyb}/H \to \mathcal{B}/H$
in the category of
topological spaces.

\vspace{2mm}
Finally, we prove
\eqref{quotientasSYZ}.
From Theorem \ref{SLAG.ppav}
\eqref{fin.sym.thm2}, the assertions
hold away from $t=0$.
What remains is to interpret
the natural $X^{\rm an}/H
\to \mathcal{B}_0/H$ as a
certain generalized version of
non-archimedean SYZ fibration.
The precise meaning is as follows.
Under the assumption that
the homomorphism $\iota : H \to {\rm GL}(g,\ZZ)$ is injective,
we can take a $H$-equivariant
Kulikov-Kunnemann model $\mathcal{X}$
over $\Delta$ (\cite[3.5 esp., (v)]{Kun}).
Then
$\mathcal{X}/H$ is automatically again a
relative minimal model of $X/H$.
We can and do use this to construct a
retraction of $X^{\rm an}/H$ to
$\mathcal{B}_0/H$ as in \cite[A.9]{Od} although
$(\mathcal{X}/H, (\mathcal{X}/H)_{t=0})$
may not be dlt in general.
See \cite[6.1.9]{BM} for a related result.
\end{proof}

Now, we are ready to
rigorously formulate and
prove
\cite[Conjecture 3]{KS}
for K-trivial
finite quotients
of abelian varieties
of any dimension,
under a slight
assumption.

\begin{Cor}[{\cite[Conjecture 3]{KS}}]
Consider an arbitrary maximally degenerating
family $(\X^*,\mathcal{L}^*)/\Delta^*$
of polarized abelian varieties of dimension $g$ and
a holomorphic action of a  group $H$ (which must be a finite group) on
$\mathcal{X}^*$ in a fiber-preserving manner such that
\begin{itemize}

\item the $H$-action fixes $\mathcal{L}_t$
 for $t\neq 0$,
 \item
the $H$-action acts trivially on
$\omega_{\mathcal{X}^*/\Delta^*}$
so that $\mathcal{X}^*/H\to
\Delta^*$ is again a relatively
K-trivial family,
\item the homomorphism $\iota : H \to {\rm GL}(g,\ZZ)$  appeared in Lemma \ref{fin.sym.lem} is injective.
\end{itemize}

Then the following two tropical
affine manifolds with singularities are isomorphic   (see  Theorem \ref{Hybrid SYZ} for the
precise meaning):
\begin{enumerate}
    \item (complex side) the limit of the base $\mathcal{B}_t/H$
    of special Lagrangian fibrations
    for $t\to 0$ with $\nabla_B(t)$
    \item \label{base2} (non-archimedean side)
    the base $\mathcal{B}_0/H$
    which underlies
    the generalized
    non-archimedean SYZ fibration.
\end{enumerate}
Further, for the latter \eqref{base2}, we further assume that for any $h$ which is not the identity element, the fixed locus of its action on $\mathcal{B}_0$ is $0$-dimensional.
Then,
the fibration
$(X/H)^{\rm an}\to \mathcal{B}_0/H$ above
($f^{\rm hyb}/H|_{t=0}$ in
Theorem \ref{Hybrid SYZ} \eqref{quotientasSYZ})
is a genuine
non-archimedean SYZ fibration in the sense of \cite{KS,NXY}
i.e., coincides with the one associated to a certain dlt minimal model
$(\mathcal{X},\mathcal{X}_0)/\Delta$.
\end{Cor}

\begin{proof}
Note that the former
affine manifold with singularities means $\mathcal{B}_t/H$
(for $t\neq 0$, or its limit for $t\to 0$)
and the latter means the $\mathcal{B}_0/H$
of Theorem \ref{Hybrid SYZ}.
Hence, we can apply Theorem \ref{Hybrid SYZ} \eqref{finite.hybrid}
and \eqref{quotientasSYZ} and their proofs. The only
remained claim is to show is the last paragraph of the statements, which amounts to prove
that for the $H$-equivariant
Mumford-Kunnemann SNC model $\mathcal{X}$,
$(\mathcal{X}/H,(\mathcal{X}/H)_0)$ is dlt.
We prove it here.

We write the irreducible decomposition of
$\mathcal{X}_0$ as $\cup_i E_i$. We want to show that
for any $h$ which is not the identity $e$,
$h$ does not fix any $E_i$ pointwise.
Suppose the contrary and take a general point of
$x\in E_i$. $\mathcal{X}$ is smooth over $\Delta$
at $x$. We take a local coordinates
$(x_1,\cdots,x_g)$ of $x\in E$ which we extend to $h$-invariant holomorphic functions
around
$x\in \X$. Then
$(x_1,\cdots,x_g,t)$ is a $H$-invariant
local coordinates of $x\in \X$, which contradicts with
nontriviality of $h$.
Hence $(\mathcal{X}/H)_0$ is reduced.

Suppose there is $h(\neq e)\in H$
which preserves a strata $Z$ of $\mathcal{X}_0$
(a log canonical center of
$(\mathcal{X},\mathcal{X}_0)$) pointwise. Then
the strata of the dual complex $\mathcal{B}_0$
which corresponds to $Z$ is fixed by $h$, hence
contradicts with our last assumption.
Therefore, the quotient
$(\mathcal{X}/H,(\mathcal{X}/H)_0)$ is dlt
(which partially extends a result by Overkamp
in the Kummer surfaces case
cf., \cite[\S 2, \S 3]{Over}).
We complete the proof.
\end{proof}

\begin{Rem}
The complementary
``prediction"s in
\cite[Conjecture 3]{KS}
are the existence of
some ``natural"
interpretation of
the non-archimedean SYZ fibration via Gromov-Hausdorff limit,
together with coincidence of critical  locus.
Note that original
claim (the way to define $\pi_{mer}$)
itself  was not
rigorous and the
formulation itself was a nontrivial problem.
Nevertheless,
our construction and
study of
$\mathcal{X}^{*,hyb}(\X)$
 and identification of its central fiber with the Gromov-Hausdorff limit of $\mathcal{X}_t$s in the proof of
Theorem \ref{Hybrid SYZ}
rigorously  formulated
their predictions and,
at the same time,
affirmatively proven
them.
\end{Rem}

\begin{Rem}[K3 surfaces case]
It is harder to consider the same problem for polarized K3 surfaces and currently we are not sure if it works.
Indeed, a subtletly exists here, as pointed out to the first author by V.Alexeev in April, 2019. That is,
in the construction of $\mathcal{B}_0$ in \cite{OO}, the appearing singularities of affine structures are always of Kodaira type i.e.,
which underlies singular fibers of minimal elliptic surfaces as well-known classical classification by Kodaira. On the other hand,
the essential skeleta \cite{AET} constructed as dual complexes of Kulikov degenerations of degree $2$ polarized K3 surfaces,
are not necessarily of Kodaira type. Hence, careful choice of Kulikov degenerations seems necessary
(even if the analogous hybrid SYZ  fibrations exist).
\end{Rem}



\section{Degenerating abelian varieties
and crystallographic groups}
\label{sec.sym}
In this section, we reveal, in an explicit manner, the relation
between the automorphism group of  families of abelian varieties and the
automorphism group of their Gromov-Hausdorff
collapses, which can be regarded as a
crystallographic group.

\begin{Setup}(after notation \S \ref{degav.notation}) \label{4.1}
At first, we consider the the family of
abelian varieties with polarizations of the
type $E={\rm diag}(e_1,\cdots,e_g)$,
following
Lemma \ref{fin.sym.lem},
i.e., $$\X^*=((\C^*)^g \times \Delta^*) /\langle p_{i,j}(t) \rangle_{1\le i,j\le g}^{\ZZ^g} \to \Delta^*,$$
for $p_{i,j}(t)
=q_{i,j}(t)^{e_i}$
where $q_{i,j}(t)
\in \C((t))^{\rm mero}$.
Recall
that the fibers are

$$\xymatrix{
\C^g/(E\ZZ^g \oplus \Omega (t)\ZZ^g)
\ar@{}[d]|{\rotatebox{90}{$\in$}}
\ar[r]^-\simeq & (\C^*)^g  /\langle p_{i,j}(t) \rangle_{1\le i,j\le g}^{\ZZ^g}=\X_t
\ar@{}[d]|{\rotatebox{90}{$\in$}}\\
\left(\overline{Z_i}\right)_{1\le i \le g} \ar@{|->}[r]&\left(\overline{\exp(2\pi \sqrt{-1}e_i^{-1} Z_i)}\right)_{1\le i \le g}
}
$$ where
$\Omega (t)=\left( \frac{1}{2\pi \sqrt{-1}} \log q_{i,j}(t)\right)_{i,j} \in \mathbb{H}_g$
and
$\{Z_i\}$ is the standard basis of $\C^g$.
Of course, $\Omega (t)$ has the ambiguity
caused by the monodromy effect. However, the lattice $(E\  \Omega (t))\ZZ^{2g}\subset \C^g$ is well-defined.
In addition, the imaginary part $${\rm Im} (\Omega (t))=\left(\frac{-1}{2\pi} \log |q_{i,j}(t)|\right)_{1\le i,j\le g}$$
is also well-defined (see the proof of Lemma
\ref{FC}).
\end{Setup}
In the following way,
we re-describe
the matrix
$B_{i,j}=({\rm val}_t q_{i,j}(t))_{i,j}\in {\rm Mat}_{g\times g}(\ZZ)$,  which appeared in  \cite{FC, Goto, Od}
(in the principally polarized case)
and the proof of Theorem
\ref{SLAG.ppav}.
\begin{Lem}\label{omega vs B}
Under the above setup \ref{4.1},
the following holds:
$$(B_{i,j})_{1\le i,j\le g}=\lim_{t\to 0} \frac{-2\pi}{\log |t|} {\rm Im} (\Omega (t))$$
\end{Lem}
\begin{proof}
The assertion simply follows from $\lim_{t\to 0} \frac{\log |q_{i,j}(t)|}{\log |t|} = {\rm val}_t q_{i,j}(t).$
\end{proof}

Below, we often denote $B_{i,j}$ simply as $B$. Now we define and study the following two  automorphism groups of $X$ in our context.
\begin{Def}
Under the above setup \ref{4.1},
\begin{enumerate}
    \item

$\Aut (X,c_1(L))$  consists of automorphisms of $X$ which preserve $c_1(L)$.
In other words, $f\in \Aut (X,c_1(L))$ preserves the flat K\"ahler metric
on each $\X_t$
with its K\"ahler class
$c_1(\cL_t)$.

   \item
$\Aut (X,L)$  consists of automorphisms of $X$ which preserve $L$.
In other words, $f\in \Aut (X,L)$ preserves
isomorphic class of the
the principal polarization $\cL_t$ on each $\X_t$.
By definition,
$\Aut (X,L)$ is a subgroup of $\Aut (X,c_1(L))$. It is also a well-known fact
that any element of
$\Aut (X,c_1(L))$ extends to
automorphism of $\mathcal{X}^*$.
\end{enumerate}
\end{Def}
Now we refine Lemma \ref{fin.sym.lem}.

\begin{Lem}\label{projection}
Under the above setup \ref{4.1},
the image of
$$l\colon
\Aut (X,c_1(L))\to {\rm GL}(g,\ZZ)$$
in Lemma \ref{fin.sym.lem} lies in
${\rm GL}(g,\ZZ) \cap {\rm O}(g,B^{-1}),$ where
${\rm O}(g,B^{-1}) :=\{M\in {\rm GL}_g(\R) \ | \ {}^t M B^{-1} M = B^{-1} \}. $
\end{Lem}
\begin{proof}
Note that $f\in \Aut (X,c_1(L))$ induces an automorphism of
$\X^*$ over $\Delta^*$,
possibly after shrinking the radius of
$\Delta^*$ by rescale,
which we denote by the same letter
$f\in \Aut (\X^*)$.
The restriction $f_t:=f|_{\X_t}$ gives
an automorphism of $(\X_t, c_1(\cL_t))$.
By the argument of Lemma \ref{fin.sym.lem},
the automorphism  $f_t$ of $\X_t=\C^g/(E\  \Omega (t))\ZZ^{2g}$ is induced from
a linear transformation $M\in {\rm GL}(g,\ZZ)$
which is independ of $t\in \Delta^*$.
Consider the principally polarized abelian variety $(\X_t,\cL_t)$ defined by $\C^g/(E\  \Omega (t))\ZZ^{2g}$.
The metric on $\X_t$ induced by $c_1(\cL_t)$ is
given by the Hermite matrix $({\rm Im} \Omega (t))^{-1}$ on $\C^g$.
Since
the automorphism $f_t$ of $\X_t=\C^g/(E\  \Omega(t))\ZZ^{2g}$ preserving the metric $c_1(\cL_t)$,
the corresponding $M\in {\rm GL}(g,\ZZ)$ satisfies the equation ${}^tM({\rm Im} \Omega (t))^{-1}M=({\rm Im} \Omega (t))^{-1}$
for any $t\in \Delta^*$.
By Proposition \ref{omega vs B}, the above equation gives the equation ${}^t M B^{-1} M = B^{-1}$.
Since such $M=l(f)$, we complete the proof.
\end{proof}
\begin{Cor}\label{no translation preserving pp}
Under the same setting as Lemma \ref{projection},
when $e_1=\cdots=e_g=1$ i.e.,
principally polarized case,
the restriction $l|_{\Aut (X,L)}$ is injective.
\end{Cor}
\begin{proof}
An automorphism $f\in \Aut (X,c_1(L))$ induces an automorphism $\tilde{f} \colon  (\C^*)^g\times \Delta^* \to (\C^*)^g\times \Delta^* $,
possibly after shrinking the radius of
$\Delta^*$ by rescale.
The automorphism $\tilde{f}$
can be described as follows.
$$\begin{array}{cccc}
\tilde{f} \colon &  (\C^*)^g\times \Delta^*      & \longrightarrow & (\C^*)^g\times \Delta^*                 \\
&\rotatebox{90}{$\in$} &                 & \rotatebox{90}{$\in$} \\
&       (z, t)             & \longmapsto     & (\tilde{f}(z,t),t)
\end{array}
$$
By Lemma \ref{projection}, we can write
$$\tilde{f}(z,t)=\tilde{f}(1,t)z^{l(f)},$$
where $$l(f)=(a_{i,j})_{1\le i,j\le g} \in {\rm GL}(g,\ZZ),$$
$$\tilde{f}(1,t):=\tilde{f}((1,\dots,1),t),$$ and
$$\tilde{f}(1,t)z^{l(f)}:=\left(\tilde{f}_i(1,t)\prod_j z_j^{a_{i,j}}\right)_i\in (\C^*)^g,$$
where $\tilde{f}_i(1,t)\in \C((t))^{\rm mero}$
are meromorphic functions on $\Delta$.
We set
$T_{\tilde{f}}(z,t):=\tilde{f}(1,t)z$. It induces a translation on each fiber $\X_t$.
In the construction of the homomorphism $p$, we considered $$T_{\tilde{f}}^{-1}\circ \tilde{f}(z,t)=z^{l(f)}.$$
If $l(f)=l(f')$, then  $T_{\tilde{f}}^{-1}\circ \tilde{f}(z,t)=T_{\tilde{f'}}^{-1}\circ \tilde{f'}(z,t)$.
That is, the equation $\tilde{f'}(z,t)=T_{\tilde{f'f^{-1}}}\circ\tilde{f}(z,t)$ holds.
The morphisms $f$ and $f'$ preserve the ample line bundle $L$.
That is, $f_t$ and $f'_t$ preserve $\cL_t$ for each fiber $\X_t$.
Hence, the morphism $T_{\tilde{f'f^{-1}}}|_t$ preserves the principal polarization $\cL_t$.
Since $\cL_t$ is principal polarization, i.e.,
the morphism $$\begin{array}{cccc}
\phi_{\cL_t} \colon &  \X_t      & \longrightarrow & {\rm Pic}^0 \X_t              \\
&\rotatebox{90}{$\in$} &                 & \rotatebox{90}{$\in$} \\
&       x            & \longmapsto     & T_x^*(\cL_t)\otimes \cL_t^{-1}
\end{array}
$$
is an isomorphism,
there is no (nontrivial) translation which
preserves the principal polarization $\cL_t$.
Therefore, the equation $\tilde{f'}(1,t)=\tilde{f}(1,t)$ holds.
It implies that $T_{\tilde{f}}=T_{\tilde{f'}}$.
Hence the equation $$\tilde{f}(z,t)=T_{\tilde{f}}(z^{l(f)})
=T_{\tilde{f'}}(z^{l(f')})=\tilde{f'}(z,t)$$
holds.
That is, the homomorphism
$l|_{\Aut (X,L)}$ is injective.
\end{proof}
Note that
the homomorphism $l:\Aut (X,c_1(L))\to {\rm GL}(g,\ZZ) \cap {\rm O}(g,B^{-1})$ ignores the effect of translations of the abelian varieties.
Now we consider a homomorphism that also reflects
translations.

\begin{Thm}\label{revised projection}
Under the same situation as Lemma \ref{projection},
there is a natural exact sequence
$$1\to
{\rm Hol}(\Delta, (\C^*)^g)\hookrightarrow
\Aut (X,c_1(L))\stackrel{p}{\longrightarrow} ({\rm GL}(g,\ZZ) \cap {\rm O}(g,B^{-1}))\ltimes \ZZ^g /B\ZZ^g,$$
where
${\rm Hol}(\Delta, (\C^*)^g)$
denotes the group of the germs of
holomorphic maps from
the neighborhood of $0\in \mathbb{C}$
to $(\mathbb{C}^\times)^g$.
\end{Thm}

Note that
via a surjective homomorphism
$$s\colon ({\rm GL}(g,\ZZ) \cap {\rm O}(g,B^{-1}))\ltimes \ZZ^g
\to
({\rm GL}(g,\ZZ) \cap {\rm O}(g,B^{-1}))\ltimes \ZZ^g /B\ZZ^g, $$
$s^{-1}({\rm Im}(p))$
is a
{\it
crystallographic group}, that is, $s^{-1}({\rm Im}(p))$ is a subgroup of the isometry group of $(\R^g, B^{-1})$
such that $s^{-1}({\rm Im}(p))$ acts discontinuously on $\R^g$
and an intersection of $s^{-1}({\rm Im}(p))$ with the group of translations on $\R^g$
is equal to $\ZZ^g$.
Later, in Theorem \ref{hybrid equivariant},
we provide a geometric meaning to the above map $p$
via ``tropicalization" i.e.,
passing to the base of
special Lagrangian fibrations or
non-archimedean SYZ fibrations.
\begin{proof}
In the proof of Corollary \ref{no translation preserving pp},
a pair $(l(f), T_{\tilde{f}})$ is
associated to
any automorphism $f\in \Aut (X,c_1(L))$.
By the definition of $\tilde{f}(1,t)$, the lift
$\tilde{f}(1,t)$ of $f(1,t)$ is uniquely determined up to $\langle p_{i,j}(t)\rangle$.
Hence, $${\rm val}_t(f(1,t)):=\overline{{\rm val}_t(\tilde{f}(1,t))}\in \ZZ^g/ B\ZZ^g$$
is well-defined.
Then we define $p(f):=(l(f),{\rm val}_t(f(1,t)))$.
It is clear that the map $p$ is a homomorphism.

To end the proof, we verify that $\ker p={\rm Hol}(\Delta, (\C^*)^g)$.
Take a lift $\tilde{f}(z,t)$ of $f(z,t)\in \ker p$.
By definition of $\ker p$,
the lift $\tilde{f}(z,t)$ can be described as $\tilde{f}(1,t)z$,
where ${\rm val}_t\tilde{f}(1,t)=\left({\rm val}_t\tilde{f}_i(1,t)\right)_i=Bv
\in B\ZZ^g$ for some $v\in \ZZ^g$.
Now we consider the morphism $F_{\tilde{f}}:\Delta^*      \to (\C^*)^g$ defined by
$$F_{\tilde{f}}(t):=
\tilde{f}(1,t)(p_{i,j}(t))^{-v}:=\left(\tilde{f}_i(1,t)\prod_j p_{i,j}(t)^{-v_j}\right)_i,$$
where $v=(v_i)\in \ZZ^g$.
It is clear that ${\rm val}_t F(t)=0\in \ZZ^g$.
In other words, $F_{\tilde{f}}(t)\in {\rm Hol}(\Delta, (\C^*)^g)$.
Then $F_{\tilde{f}}(t)$ does not depend on how $\tilde{f}$ is taken.
Hence  $F_f$ denotes $F_{\tilde{f}}$.
Indeed, the lift $\tilde{f}(1,t)$ of $f(1,t)$ is uniquely determined up to $\langle p_{i,j}(t)\rangle$.
Any other lift $\tilde{f'}(z,t)$ of $f(z,t)$ can be described as $\tilde{f}(z,t)(p_{i,j}(t))^w$ for some $w\in \ZZ^g$.
Consider $F_{\tilde{f'}}(t)$ in the same manner with respect to $\tilde{f'}$. Since ${\rm val}_t\tilde{f'}(1,t)=Bv+Bw=B(v+w)$,
it holds that
\[\begin{split}
    F_{\tilde{f'}}(t)&=\tilde{f'}(1,t)(p_{i,j}(t))^{-(v+w)}=\tilde{f}(1,t)(p_{i,j}(t))^w (p_{i,j}(t))^{-(v+w)} \\
&=\tilde{f}(1,t)(p_{i,j}(t))^{-v}=F_{\tilde{f}}(t).
\end{split}\]
Hence, we obtain the homomorphism $\varphi:\ker p \to  {\rm Hol}(\Delta, (\C^*)^g)$ defined by $f\mapsto F_f(t)$.
On the other hand, for any $F(t)\in  {\rm Hol}(\Delta, (\C^*)^g)$, the morphism $F(t)z:(\C^*)^g\times \Delta^*\to(\C^*)^g\times \Delta^*$
descends to an automorphism $f_F\in \ker p$.
That is, we obtain the homomorphism $\psi : {\rm Hol}(\Delta, (\C^*)^g)\to \ker p$
defined by $F\mapsto f_F$.
It is obvious that $\varphi\circ \psi={\rm id}$.
By construction of $\varphi$, for any $f\in \ker p$, we can take $F_f(t)z$ as a lift of $f$.
It implies that $\psi\circ\varphi={\rm id}$ also holds.
Therefore, $\ker p\cong{\rm Hol}(\Delta, (\C^*)^g)$ holds.
\end{proof}
Now we consider more geometric meaning of $({\rm GL}(g,\ZZ) \cap {\rm  O}(g,B^{-1}))\ltimes \ZZ^g /B\ZZ^g$.
For the family $(\X^*,c_1(\cL^*))$ under discussion,
set $$\mathcal{B}_{0}:=(\mathcal{B}_0,\nabla_A(0),\nabla_B(0),g_0)$$ as Theorem \ref{SLAG.OO}.
Then Theorem \ref{SLAG.OO} implies that $(\mathcal{B}_{0},\nabla_B(0),g_0)$ is the
integral affine manifold $\R^g/B\ZZ^g$ with its lattice points $\ZZ^g/B\ZZ^g$ and the flat metric induced by $B^{-1}$.
See \cite{GH} or \cite{GS06} for a series of definitions on (integral) affine manifolds.
Actually, $\nabla_A(0)$ is determined by $(\mathcal{B}_{0},\nabla_B(0),g_0)$ via Legendre transform (c.f. \cite{Hit}).
For this reason, we denote
$(\mathcal{B}_0,\nabla_A(0),\nabla_B(0),g_0)$
as $(\mathcal{B}_{0},\nabla_B(0),g_0)$, that is, $(\R^g/B\ZZ^g,B^{-1})$.

\begin{Def}\label{IAM.def}
For the flat integral affine manifold $(\R^g/B\ZZ^g,B^{-1})$ as above,
define the automorphic group $\Aut (\R^g/B\ZZ^g, B^{-1})$ to be the group consisting of automorphisms that preserve the structure of $(\R^g/B\ZZ^g,B^{-1})$.
That is, an element $f\in\Aut (\R^g/B\ZZ^g, B^{-1})$ is an integral affine map
$f: \R^g/B\ZZ^g \stackrel{\sim}{\longrightarrow} \R^g/B\ZZ^g $ that preserves the flat  metric induced by $B^{-1}$.
\end{Def}

\begin{Prop}\label{Aut of GH}
Under the notation of Definition
\ref{IAM.def}, the following holds
$$\Aut (\R^g/B\ZZ^g, B^{-1})\cong
({\rm GL}(g,\ZZ) \cap {\rm O}(g,B^{-1}))\ltimes \ZZ^g/ B\ZZ^g.$$

\end{Prop}
\begin{proof}
Consider the universal covering $u:\R^g \twoheadrightarrow \R^g/B\ZZ^g$
which we denote by $x\mapsto \overline{x}$.
Note that the map $u$ is an integral affine map.
For $f\in\Aut (\R^g/B\ZZ^g, B^{-1})$, the map $f\circ u$ is also  an universal covering.
The universality gives an homeomorphism $\tilde{f}:\R^g\stackrel{\sim}{\longrightarrow}\R^g$ such that $u\circ \tilde{f}=f\circ u$.
Here, the map $\tilde{f}: (\R^g,B^{-1}) \stackrel{\sim}{\longrightarrow} (\R^g,B^{-1})$ is an isometry since
the map $f$ is an isometry and the map $u$ is locally trivial.
Hence,
\[\tilde{f}(x)=Mx+v,\]
where $M\in O(g,B^{-1})$ and $v\in \R^g$.
Moreover, since the map $f$ is an integral affine morphism and
the map $u$ is locally trivial,
the map $\tilde{f}$ is also an integral affine morphism.
Hence, the above pair $(M,v)$ satisfies $M\in {\rm GL}(g,\ZZ)$ and $v\in \ZZ^g$.
It implies that $f(\overline{x})=M\overline{x}+\overline{v}$, where $(M,\overline{v})\in
({\rm GL}(g,\ZZ) \cap {\rm O}(g,B^{-1}))\ltimes \ZZ^g/ B\ZZ^g$.
That is, we obtain the homomorphism $$\varphi: \Aut (\R^g/B\ZZ^g, B^{-1})\to
({\rm GL}(g,\ZZ) \cap {\rm O}(g,B^{-1}))\ltimes \ZZ^g/ B\ZZ^g.$$
Since the existence of the inverse map is obvious, the assertion follows.
\end{proof}

In particular, $\Aut (X,c_1(L))$ induces at most finite group actions on $\mathcal{B}_0$.
In Theorem \ref{Hybrid SYZ} \eqref{finite.hybrid}, we have already seen that
$\Aut (X,c_1(L))$ induces continuous actions on
 both $\X^{*, \rm hyb}$ and $\mathcal{B}$.
Further,
$f^{\rm hyb}$ is $\Aut (X,c_1(L))$-equivariant.
Now we describe the actions of $\Aut (X,c_1(L))$ concretely,
by relating the $H$-action of
Theorem \ref{Hybrid SYZ}
\eqref{finite.hybrid}
and the map $p$ in
Theorem \ref{revised projection}.
\begin{Thm}\label{hybrid equivariant}
Under the same situation as Theorem \ref{revised projection},
for any subgroup $H$ of $\Aut (X,c_1(L))$,
consider the induced actions on $\X^{*, \rm hyb}$ and
$\mathcal{B}$
of $H$ as Theorem \ref{Hybrid SYZ}.
We denote the restriction of the latter action
on $\mathcal{B}_t$ as $p_t(h)$
for each $h\in H$ and $t\in\Delta$.

Then $p_t(h)$ is explicitly described as
\begin{enumerate}
    \item \label{action1}
    For $t\neq 0$,
$$\left(l(h),\log_{|t|}|\tilde{h}(1,t)|\right)\in
{\rm GL}(g,\ZZ)\ltimes \left(\R^g/\left(\frac{-2\pi}{\log |t|} {\rm Im} \Omega (t)\right)\ZZ^g \right),$$
where $\tilde{h}(1,t)$ is a lift of $h(1,t)$
and $\log_{|t|}(-)$ simply means the
logarithm with the base $|t|$. Note that the ambiguity of the choice of $\tilde{h}(1,t)$ vanishes on $\mathcal{B}_t$.

\item \label{action2}
$p_0(h)$ on $\mathcal{B}_0$ is equal to $p(h)$ as constructed in Theorem \ref{revised projection}.
\end{enumerate}
Further, $f_t(hx)=p_t(h)f_t(x)$
for each $h\in H$ and $t\in \Delta$.
\end{Thm}
\begin{proof}
First we show \eqref{action1}.
By the proof of
Theorem \ref{revised projection}, $h\in H$ is lifted to
\[\tilde{h}(z,t)=\psi (t) t^v z^{M} :(\C^*)^g\times \Delta^*\to(\C^*)^g ,\] where $(M,v)=p(h)$ and
$\psi(t):=\tilde{h}(1,t)/t^v\in {\rm Hol}(\Delta, (\C^*))$.
By the argument of Lemma \ref{projection}, $M=l(h)$ induces the automorphism of
\[(\mathcal{B}_t,\nabla_B(t),g_t)\cong \left(\left(\R^g/\left(\frac{-2\pi}{\log |t|} {\rm Im} \Omega (t)\right)\ZZ^g \right),
\left(\frac{-2\pi}{\log |t|} {\rm Im} (\Omega (t))\right)^{-1}\right)\] for $t\neq 0$.
Hence, $\left(l(h),\log_{|t|}|\tilde{h}(1,t)|\right)$ also induces the automorphism of $(\mathcal{B}_t,\nabla_B(t),g_t)$.
Note that the lift $\tilde{h}(1,t)$ of $h(1,t)$ is uniquely determined up to $\langle p_{i,j}(t)\rangle$.
It implies that $\log_{|t|}|\tilde{h}(1,t)|$ is well-defined on $\R^g/
\left(\frac{-2\pi}{\log |t|} {\rm Im} \Omega (t)\right)=\R^g/
\left(\log_{|t|} |p_{i,j}(t)|\right)$.

Now consider $f_t(hx)$.
By the argument of the proof of Theorem \ref{Hybrid SYZ}, special Lagrangian fibration
$f_t(x=(x_1,\cdots,x_g))$ is described for $t\neq 0$ as
$$f_t(x)=\left(-\log_{|t|}|x_i|\right)_{i=1,\cdots,g}  \in \R^g/
\left(\frac{-2\pi}{\log |t|} {\rm Im} \Omega (t)\right)\ZZ^g. $$
Hence, it is clear that $$f_t(hx)=l(h)f_t(x)+\log_{|t|}|\tilde{h}(1,t)|=\left(l(h),\log_{|t|}|\tilde{h}(1,t)|\right)\cdot f_t(x).$$
By Theorem \ref{Hybrid SYZ}, $f_t$ is $H$-equivariant.
It implies that $\left(l(h),\log_{|t|}|\tilde{h}(1,t)|\right)$ is equal to the induced action $p_t(h)$ on $\mathcal{B}_t$ for $t\neq 0$.

Now we prove \eqref{action2} by
considering $f_0(hx)$.
Now, $f_0(x)$ is defined as $(-\log |Z|_x)\in \R^g/B\ZZ^g$, where $x\in X^{\rm an}$, $|\cdot|_x$ means the corresponding multiplicative seminorm on $X$ and $Z=(Z_i)$ is the coordinates of the split algebraic torus as appeared in \cite{Goto}.
Then, it follows from $ \log |\psi(t)|_x =0$ that $f_0(hx)=p(h)f_0(x)$.
Thus, $p_0(h)=p(h)$ holds similarly.

Finally, the $H$-equivalence of $f^{\rm hyb}$
implies the last claim of the theorem.
\end{proof}

In general, we cannot expect that the symmetry of $\mathcal{B}_0$ lifts on the symmetry of $(X,c_1(L))$ or
$\mathcal{X}^*$.
However,  we shall see that such lifting exists
for the following special situation.
\begin{Setup}\label{Setup2}
We start with a
 positive definite symmetric matrix $B=(b_{i,j})_{i,j}\in {\rm Mat}_{g\times g}(\ZZ)$.
Then we obtain the following family
$$\X_B^*=(\Delta^*\times (\C^*)^g) /\langle t^B \rangle_{i,j} \to \Delta^*,$$
where $t^B$ means the $g\times g$
matrix
$\left( t^{b_{i,j}}\right)_{i,j}$.
We denote the associated polarized smooth variety to
$(\X_B^*,\mathcal{L}_B^*)$
by $(X_B,L_B)$.
\end{Setup}
In this situation, we can strengthen Theorem \ref{revised projection} as a split exact sequence.
\begin{Prop}\label{special case}
Under the above setup \ref{Setup2}, we have
$$\Aut (X_B,c_1(L_B))\cong \Aut (\mathcal{B}_0)\ltimes {\rm Hol}(\Delta, (\C^*)^g).$$
\end{Prop}
\begin{proof}
By Theorem \ref{revised projection},
it suffices to see that there is a homomorphism $\iota: \Aut (\mathcal{B}_0)\to \Aut (X_B,c_1(L_B))$
such that $p\circ\iota={\rm id}$.
For $(M,\overline{v})\in \Aut (\mathcal{B}_0)\cong
({\rm GL}(g,\ZZ) \cap {\rm O}(g,B^{-1}))\ltimes \ZZ^g/ B\ZZ^g$,
consider the map $\tilde{f}: (\C^*)^g\times \Delta^* \to (\C^*)^g\times \Delta^*$ defined by $\tilde{f}(z,t):=t^vz^M$ in the same
manner of the proof of Corollary \ref{projection}.
Here, we see that the map $\tilde{f}$ descends to the map $f:(\X_B^*, c_1(\cL_B^*))\to (\X_B^*,c_1(\cL_B^*))$.
Indeed, the fiber $\X_B^*|_t$ is given by $\C^g/(I \ \Omega(t))\ZZ^{2g}$ for each $t\in \Delta^*$, where
$\Omega(t)=\frac{\log t}{2\pi\sqrt{-1}}B.$
The above $M\in {\rm GL}(g,\ZZ) \cap {\rm O}(g,B^{-1})$ satisfies $MB\ZZ^g=B\ZZ^g$ and ${}^tMB^{-1}M=B^{-1}$.
In particular, it implies that ${}^tM({\rm Im} \Omega (t))^{-1}M=({\rm Im} \Omega (t))^{-1}$.
Hence the automorphism $z^M: \X_B^* \to \X_B^*$ is well-defined and preserves the metric for each fiber.
Since translations have no effects on the period matrix and the metric,
$\tilde{f}$ descends to  the automorphism $f$ of $(X_B,c_1(L_B))$.
Note that the morphism $f$ does not depend on how $\tilde{f}$ is taken.
Indeed, the ambiguity of $v$ vanishes in the process of obtaining $f$ from $\tilde{f}$.
Hence we obtain the homomorphism $\iota: \Aut (\mathcal{B}_0)\to \Aut (X_B,c_1(L_B))$.
By construction, it is obvious that $p\circ \iota={\rm id}$.  Therefore, the assertion holds.
\end{proof}

\begin{Rem}
We only considered
maximal degenerations case in this paper but  degenerations of $g$-dimensional polarized  abelian varieties $(X,L)/\C((t))^{\rm mero}$
with other torus rank $i$ should work similarly.
Firstly, the same method  as
\cite[proof of 1.1, cf., also Remark 5.3]{Mat}
by studying actions on
weight filtration
or weight spectral sequences, implies that for any polarized endomorphism
$f$ of
$(X,L)/\C((t))^{\rm mero}$,
any eigenvalue $e$ of $f^*|_{H^1(\mathcal{X}_t,\C)}$,
the degree $d$ of the minimal polynomial of rational coefficients satisfies
$d\ge \max\{i,2g -2i\}.$
 We may explore
more details in future.
\end{Rem}


\section*{Appendix: non-archimedean Calabi-Yau metrics \\ vs. Calabi-Yau metrics}

In this appendix, we show that non-archimedean Calabi-Yau metric appears as the limit of (complex)
Calabi-Yau metrics,
at least in the case of abelian varieties.
This answers a question of Yang Li asked after the appearance of this paper.
We thank him for a nice question.

\vspace{2mm}
First we recall that any maximal degeneration of polarized abelian varieties $(\mathcal{X}^*,\mathcal{L}^*)\to \Delta^*$
can be written in the following way, after Lemma \ref{fin.sym.lem}.
Suppose that the polarizations are of the type $(e_1,\cdots,e_g)$ with $e_i|e_{i+1}$
and we set
$E={\rm diag}(e_1,\cdots,e_g)$. Then we can write
$$\X^*=\sqcup_{t\in \Delta^*}X_t=((\C^*)^g \times \Delta^*) /\langle p_{i,j}(t) \rangle_{1\le i,j\le g} \to \Delta^*,$$
for some $p_{i,j}(t)
=q_{i,j}(t)^{e_i}$
where $q_{i,j}(t)
\in \C((t))^{\rm mero}$. We denote the polarizations as $\mathcal{L}^*=\sqcup_{t\in \Delta^*} \mathcal{L}_t$.
Recall
that the fibers are
$$\C^g/(E\ZZ^g \oplus \Omega (t)\ZZ^g)$$ where
$
\Omega (t)=\left( \frac{1}{2\pi \sqrt{-1}} \log q_{i,j}(t)\right)_{i,j}=\alpha_{i,j}(t)+\sqrt{-1}\beta_{i,j}(t)$.
The polarization $\mathcal{L}_t$ is obtained as
a discrete quotient of the trivial line bundle over $\mathbb{C}^g$. That is, for
$$\tilde{\mathcal{L}_t}:=\C\times \C^g,$$
as a line bundle over $\C^g$, we set
$$\mathcal{L}_t=\tilde{\mathcal{L}_t}/(E\mathbb{Z}^g+\Omega(t)\mathbb{Z}^g),$$
where the action of $(E\mathbb{Z}^g+\Omega(t)\mathbb{Z}^g)$
is defined by the automorphic factor
\begin{align}
\label{23}
f(E\vec{m}+\Omega\vec{n})=&(e^{\pi z(\beta_{i,j}(t))^{-1}\overline{u}+\frac{1}{2}\pi u(\beta_{i,j}(t))^{-1}\overline{u}}\cdot (-1)^{{}^{t}\vec{m}\cdot \vec{n}})\\
\label{24}
&\times e^{-(\pi z(\beta_{i,j}(t))^{-1} u                +\frac{1}{2}\pi u(\beta_{i,j}(t))^{-1} u )}\in \mathbb{C}^*
\end{align}
for $z\mapsto z+u$ with $u=E\vec{m}+\Omega\vec{n}$.
The former \eqref{23} is the Appel-Humbert theorem description in e.g. \cite{Mum}, while the
 latter \eqref{24}
 is just a normalization factor (coboundary in the context of group cohomology),
 which does not change the isomorphic class of the line bundle.
The generic fiber of $(\X^*,\mathcal{L}^*) \to \Delta^*$
is denoted as $(X,L)$ over $\C((t))^{\rm mero}$ and its analytification is $(X^{\rm an},L^{\rm an})$.

\vspace{3mm}
Here, we introduce a limiting process of metrics from complex (Archimedean) to non-archimedean
world.

\begin{Def}
For a polarized projective morphism $\mathcal{X}^*=\cup_t X_t \to \Delta^*$ with a relative
ample line bundle $\mathcal{L}^*$ on $\mathcal{X}^*$, we denote the
associate non-archimedean analytification over $\mathbb{C}((t))^{\rm mero}$ as
$(X,L)$. We take the limit hybrid analytification (\cite[Definition 4.9]{BJ})
$X^{\rm hyb}=\mathcal{X}^*\sqcup X$ and the associated line bundle $L^{\rm hyb}$ over
$X^{\rm hyb}$.

For a continuous family $h_t$ of hermitian metrics on $L_t:=\mathcal{L}|_{X_t}$,
suppose there is a metric $h$ of $L$ such that
$\{h_t\}_{t\neq 0}$ together with $h$ gives a {\it continuous} family of metrics on
the hybrid analytification $L^{\rm hyb}$.
Then we call that the non-archimedean metric $h$ is the
{\it normalized limit metric} of $h_t$s.
\end{Def}

After this work, the authors learnt there is an interesting
then-ongoing work by L.\ Pille-Schneider \cite{PS}
which perhaps is related to the above notion but more general.
Now, recall the so-called cubic metric $h_t$ (\cite{MB})
on $L_t$ for each $t\in \Delta^*$, which is unique up to multiplication by
positive real constants. These are characterized as hermitian metrics
whose curvatures forms $\omega_{h_t}$ are flat K\"ahler forms.
We show how this $h_t$ limits to the non-archimedean CY metric on $L^{\rm an}$, which exists by \cite{BFJ} and denoted by
$h^{\rm NA}$ in this notes.

\begin{Thm}\label{NAMA as lim}
Appropriately rescaled (i.e., constants multiplied for each $t$)
cubic metrics $h_t$ has the normalized limit metric which is
the non-archimedean CY metric $h^{\rm NA}$.
\end{Thm}

\begin{proof}
Note that $\tilde{\mathcal{L}_t}$ is descended to $\mathbb{C}^g/E\mathbb{Z}^g\simeq (\mathbb{C}^*)^g$ since
$f(E\vec{m})=1$ for any $\vec{m}\in \mathbb{Z}^g$ by the above definition. We denote it by $\mathcal{L}'_t$, which is trivial line bundle.
This will converge to a line bundle $L'^{\rm an}$ which descends to $L^{\rm an}$.
By the compatible trivialization of both
$\mathcal{L}'_t$ and $L'^{\rm an}$, we take a non-vanishing holomorphic (constant) section $1$.

Then, since $\sqrt{-1}\partial \overline{\partial} \log h_t(1)=
\sqrt{-1}\sum_{1\le i,j\le g}\beta_{i,j}dz_i\wedge d\overline{z_j}$,
we have
\begin{align}
-\log h_t(1)
&={}^{t}\vec{z}\sum_{1\le i,j\le g}(\beta_{i,j}(t))^{-1}\overline{\vec{z}}-{\rm Re}({}^{t}\vec{z}\sum_{1\le i,j\le g}(\beta_{i,j}(t))^{-1}\vec{z})\\
&=\sum_{1\le i,j\le g}y_i(\beta_{i,j}(t))^{-1}y_j,
\end{align}
where $z_i=x_i+\sqrt{-1}y_i$, after suitably multiplying $h_t$ by real constants.

On the other hand,
by \cite[Theorem 4.3]{Liu}
\footnote{note that this in particular shows the ``NA MA-real MA comparison property" in the sense of
\cite{YangLi} in the abelian varieties case.},
we have
$$-\log h_{\rm NA}(1)=\sum_{1\le i,j\le g}y_i (B_{i,j}(t))^{-1}y_j,$$
again after suitably multiplying $h_{\rm NA}$ by a real constant.
Thus, by the assertion by Lemma \ref{omega vs B}.
\end{proof}

\begin{Rem}
Analogously, it is confirmed that the normalized limit metric of appropriately rescaled
Fubini-Study (type) metrics is non-archimedean Fubini-Study metric
\cite[Section 3.6]{YangLi}. Depending on the observation, the same paper also implicitly proved
the analogue as Theorem \ref{NAMA as lim} for general non-archimedean CY metric
holds, at least over a large open subset
by \cite[(Lemma 4.1$+)$Theorem 4.7]{YangLi}, under the assumption of ``NA MA - real MA comparison property" (3.11 of {\it op.cit}).
Indeed, note that once we locally trivialize $\mathcal{L}$ of {\it op.cit} around $E_J$ then $\phi_{J,t}$ and
(resp., $\phi_{CY,t}$) are taken as $-\log |1|_{FS,t}$ (resp., $-\log |1|_{CY,t}$).
\end{Rem}

\begin{Ques}
Is Theorem \ref{NAMA as lim} (globally)
true for any maximal degeneration of general Calabi-Yau varieties?
\end{Ques}

Another related interesting question is
\begin{Ques}\label{galaxy.problem}
Is the non-archimedean Calabi-Yau metric always
``geometrically realizable" by galaxy model
in the sense of \cite[\S 2]{infinite}?
\end{Ques}
Recall that galaxy model is a certain model over the Puiseux series ring with infinite components of the reduction in general.
See \cite[Section 2.6]{infinite} for the precise details.
The case of abelian varieties is again confirmed (essentially by W.Gubler \cite{Gub, Gub10}).
Also, under the same assumption
(\cite[3.11]{YangLi}),
the answer to Q \ref{galaxy.problem} is yes at least over a large open subset
(\cite{Od22}).


\vspace{2mm}
\noindent
{\it Contact}: \\
\noindent
{\tt k.goto@math.kyoto-u.ac.jp},
{\tt yodaka@math.kyoto-u.ac.jp} \\
Department of Mathematics, Kyoto university, Kyoto, Japan \\

\end{document}